\documentclass[leqno,12pt]{amsart}
\usepackage [dvips]{graphicx}
\usepackage{epsf}
\usepackage{amssymb,amsmath,amssymb,xy,verbatim,amscd}
\usepackage{amsfonts}

\usepackage{rotating}

\theoremstyle{plain}
\newtheorem{theorem}{Theorem}[section]
\newtheorem{proposition}[theorem]{Proposition}
\newtheorem{lemma}[theorem]{Lemma}

\theoremstyle{definition}
\newtheorem{definition}[theorem]{Definition}
\newtheorem{corollary}[theorem]{Corollary}
\newtheorem{example}[theorem]{Example}
\newtheorem{remark}[theorem]{Remark}

\numberwithin{equation}{theorem}

\newcommand{\C}{{\mathbb C}}
\newcommand{\R}{{\mathbb R}}
\newcommand{\Z}{{\mathbb Z}}

\newcommand{\Q}{{\mathbb Q}}
\renewcommand{\ll}{{\langle}}
\newcommand{\rr}{{\rangle}}
\newcommand{\reg}{{\mathrm{reg}}}
\newcommand{\CR}{{\mathcal{B}}}

\newcommand{\CH}{{\mathcal{H}}}
\newcommand{\CB}{{\mathcal{B}}}
\newcommand{\CG}{{\mathcal{G}}}

\newcommand{\CT}{{\mathcal T}}

\newcommand{\SU}{\operatorname{SU}}

\newcommand{\Res}{\operatorname{Res}}

\renewcommand{\ll}{{\langle}}

\newcommand{\la}{{\langle}}
\newcommand{\ra}{{\rangle}}

\newcommand{\vol}{{\mathrm{vol}}}
\newcommand{\Trans}{{T}}

\renewcommand{\CR}{{\mathcal{R}}}
\newcommand{\CS}{{\mathcal{S}}}
\newcommand{\CM}{{\mathcal{M}}}

\newcommand{\AV}{{A}}

\newcommand{\SZ}{{Z}}

\newcommand{\frh}{{\mathfrak h}}

\newcommand{\frg}{{\mathfrak g}}
\newcommand{\frt}{{\mathfrak t}}

\newcommand{\frA}{{\mathfrak A}}

\def\barra{\vrule height10pt depth4pt width0pt}

\begin{document}
%title{ Computation of Witten series \\and volumes of moduli spaces \\
%of flat bundles over surfaces.}

\title[Bernoulli series and volumes of moduli spaces]{Multiple Bernoulli series  \\and volumes of moduli spaces \\
of flat bundles over surfaces.}

\author{Velleda Baldoni, Arzu Boysal and Mich\`ele Vergne}
\date{}

\begin{abstract}
Using Szenes formula for multiple Bernoulli series, we explain how
to compute  Witten series associated to classical Lie algebras.
Particular instances of these series compute  volumes of moduli
spaces of flat bundles over surfaces, and also certain multiple zeta
values.
\end{abstract}
\maketitle

{\small \tableofcontents}

\newpage

\section*{Introduction}
Let $V$ be a finite dimensional real vector space, and $\Lambda$ a
lattice in $V$.  We denote the dual of $\Lambda$ by $\Gamma$.

We consider a finite sequence of vectors $\Phi$ lying in $\Lambda$,
and let $\Gamma_{\reg}(\Phi)=\{\gamma\in \Gamma |\; \ll \phi,\gamma
\rr \neq 0,\;\mbox{for all} \;\phi\in \Phi\}$ be the set of regular
elements in $\Gamma$ relative to $\Phi$.

In this paper we compute

\begin{equation}\label{eq:B}
\mathcal{B}(\Phi, \Lambda)(v)=\sum_{\gamma \in \Gamma_{\reg}(\Phi)}
\frac{e^{\ll 2i\pi v ,\gamma \rr}}{ \prod_{\phi \in \Phi} \ll 2i\pi
\phi,\gamma \rr}, \end{equation} a function  on the torus
$V/\Lambda$. This sum, if not absolutely convergent, has a meaning
as a generalized function. If $\Phi$ generates $V$, then
$\mathcal{B}(\Phi,\Lambda)$ is \textit{piecewise polynomial} (see \cite{sze1}).
%(cf. Proposition \ref{prop:polyn}).

\medskip

For example, for $V=\R e_1\oplus \R e_2$ with standard lattice $\Lambda=\Z e_1+\Z e_2$,
if we choose  $\Phi=[e_1,e_1,e_2, e_1+e_2,e_1-e_2]$, then
$$\mathcal B(\Phi,\Lambda)(v_1 e_1+v_2 e_2)=\sum'_{n_1,n_2} \frac{e^{2i\pi (v_1 n_1+v_2 n_2)}}{(2i\pi n_1)^2(2i\pi n_2) (2i\pi (n_1+n_2)) (2i\pi (n_1-n_2))},$$
where the summation $\sum'$ means that we sum only over the integers
$n_1$ and $n_2$ such that $n_1n_2(n_1+n_2)(n_1-n_2)\neq 0$. The
expression for $\mathcal B(\Phi,\Lambda)(v_1e_1+v_2 e_2)$ as a
piecewise polynomial function  of $v_1$ and $v_2$ (of degree $5$) is
given in Section \ref{section:clroot}, Equation (\ref{eq:BB2}).

\bigskip

We call $\mathcal{B}(\Phi, \Lambda)$ the {\it multiple Bernoulli series}
associated to $\Phi$ and $\Lambda$.  Multiple Bernoulli series have
been  extensively studied by A. Szenes (\cite{sze1},\cite{sze2}). They are natural  generalizations of Bernoulli series:
for $V=\R \omega$, $\Lambda=\Z \omega$ and $\Phi_k=[\omega,\omega,\ldots, \omega]$,
where  $\omega$ is repeated $k$ times with $k>0$, the function $$\mathcal{B}(\Phi_k,\Lambda)(t\omega)=\sum_{n\neq 0,\;n\in \Z} \frac{e^{2i\pi nt}}{
(2i\pi n)^k}$$ is equal to $-\frac{1}{k!} B(k,\{t\})$ where
$B(k,t)$ denotes the $k^{\text{th}}$ Bernoulli polynomial in
variable $t$, and $\{t\}=t-[t]$ is the fractional part of $t$.
If $k=2g$ and $t=0$,  due to the symmetry $n\to -n$,
$$\mathcal{B}(\Phi_{2g},\Lambda)(0)=2\frac{1}{(2i\pi)^{2g}}\zeta(2g).$$
From the residue theorem in one variable, for $k>0$,   $$\sum_{n\neq 0,\;n\in \Z} \frac{e^{2i\pi nt}}{
(2i\pi n)^k}={\rm Res}_{z=0}(
\frac{1}{z^k} e^{zt}\frac{1}{1-e^z}).$$

Szenes multidimensional residue formula (see Theorem
\ref{theo:main}) is the generalization  of this formula to higher
dimension, and it is the tool that we use for computing
$\mathcal{B}(\Phi, \Lambda)(v)$ as a piecewise polynomial function.

\medskip

A particular but crucial instance of multiple Bernoulli series is
when $\Lambda$ is the coroot lattice of a compact connected simple Lie group $G$, and  $\Phi$
is comprised of positive coroots of $G$.   The series
$\mathcal{B}(\Phi_{2g-2},\Lambda)$, where the argument $\Phi_{2g-2}$  refers to taking elements of $\Phi$ with
multiplicity $2g-2$, appeared in the work of E.~Witten (\cite{wi}, \S 3), where Witten shows that its value at $v=0$
is (up to a scalar depending on $G$ and $g$) the symplectic volume of the moduli space of flat
$G$-connections on a Riemann surface of genus $g$.
Similarly, for a regular element $v$ of the Cartan Lie algebra of $G$,
Witten shows that the value of $\mathcal{B}(\Phi_{2g-1},\Lambda)(v)$ is (up to a scalar depending on $G$ and $g$) the
symplectic volume of the moduli space of flat $G$-connections on a
Riemann surface of genus $g$ with one boundary component, around which the holonomy is determined by $v$.

More generally, for the above choice of $\Lambda$ and $\Phi$, when
${\bf v}=\{v_1,\ldots,v_s\}$ is a collection of $s$ regular elements
of the Cartan Lie algebra, certain linear combinations of
$\mathcal{B}(\Phi_{2g-2+s},\Lambda)$ at some particular values
(depending on ${\bf v}$) is the symplectic volume of the moduli
space of flat $G$-connections on a Riemann surface of genus $g$ with
$s$ boundary components, around which the holonomy is determined by
$\bf v$. Then, its dependance on $\bf v$ is piecewise polynomial.

\bigskip

Multiple Bernoulli series have also been studied by P.E.~Gunnells and R.~Sczech (\cite{gs}) in view of applications to zeta
functions of real number fields.  Explicit computations of volumes of moduli spaces of flat bundles on Riemann surfaces
are also obtained in \cite{gs}.  The techniques they used is a generalization of the continued fraction algorithm and is
different from ours.

\bigskip

Y. Komori, K. Matsumoto and H. Tsumura
(\cite{kmt0},\cite{kmt1},\cite{kmt2}) studied the restriction of the
series (\ref{eq:B}), by summing it over the cone of dominant regular
weights of a semi-simple Lie group $G$, and defined  functions
$\zeta({\bf s}, v, G)$ (cf.~ Section \ref{comp}). They also obtained
relations between these functions over $\Q$. When $\Lambda$ is the
coroot lattice of a compact connected simple Lie group $G$ and  the
sequence $\Phi$ is the set of its positive coroots with equal even
multiplicity for long roots and (possibly different) equal even
multiplicity for short roots, due to the Weyl group symmetry, the
summation $\mathcal{B}(\Phi,\Lambda)(0)$ over the full (regular)
weight lattice is just (up to multiplication by an appropriate power
of $(2\pi)$) Komori-Matsumoto-Tsumura  zeta function $\zeta({\bf
s},0,G)$. Thus, the value of $\zeta({\bf s},0,G)$  (up to a certain
power of  $(2\pi)$) is a rational number which can be computed
explicitly, and we give examples of such computations.

As it is observed in \cite{kmt0},
some instances of the series $\zeta({\bf s}, v, G)$ also compute certain multiple zeta values.  In the last part of the article we give various
such computations of multiple zeta values using $\mathcal{B}(\Phi,\Lambda)$.

\bigskip

Here is the outline of individual sections.

\medskip

In Section 1, we recall a formula due to A. Szenes, which allows an efficient computation of
$\mathcal{B}(\Phi,\Lambda)$.

In Section 2,
we give an outline of an algorithm that efficiently computes the needed ingredients of this formula for
classical root systems. We also give several simple examples.

In Section 3, we show how this applies  to the symplectic volume of the moduli space of flat $G$-connections on
a Riemann surface of genus $g$ with  $s$ boundary components.  We obtain  an expression for the symplectic volume
by taking the limit of the Verlinde formula.  We then show that our formula thus obtained coincides with
that of Witten (including the constants) given in terms of the Riemannian volumes of $G$ and $T$.
We also give examples of these functions.

In Sections 4 and 5, we give several examples and tables of Witten volumes, which include some examples from
\cite{kmt0}, \cite{kmt1} and \cite{kmt2}. We give an idea of  computational limitation of our algorithm
(written as a simple Maple program) in terms of the rank of the group $G$ and the number of elements in $\Phi$.
Following Y. Komori, K. Matsumoto, H. Tsumura, we also give some examples of rational
multiple zeta values.

To compute more examples, our Maple program is available
on the webpage of the last author.

Finally, in the appendix, for completeness, we include a proof of
Szenes formula.

\subsection*{Acknowledgements}  Part of this work was completed when all three authors
were at Mathematisches Forschunginstitut Oberwolfach as a part of
Research in Pairs programme in February 2012. We would like to
express our gratitude to the institute for their hospitality.

The first author was partially supported by a PRIN$2009$ grant, the
second author was partially supported by Bo\u{g}azici University (B.U.)
Research Fund $5076$.  The third author wishes to thank Bo\u{g}azici
University for support of research visits.

We thank Shrawan Kumar for pointing out a minor mistake in the
earlier version of this manuscript in the formula of Proposition
\ref{P:factor} for the case of Lie group of type $G_2$.

We also thank the referees for their suggestions on the long abstract version of this manuscript.

\newpage
\section*{Notations}
\[\begin{array}{ll}
U &\mbox{r-dimensional real vector space.}\\
V &\mbox{dual of $U$; $v \in V$.}\\
\ll \;,\; \rr &\mbox{the pairing between $U$ and $V$.}\\
\Gamma &\mbox{a lattice in $U$; $\gamma\in \Gamma$.}\\
\Lambda:=\Gamma^* &\mbox{dual lattice in $V$; $\ll \Gamma,\Lambda\rr\subset \Z$, $\lambda\in \Lambda$.}\\
\Phi &\mbox{a sequence of  vectors in $V$; $\phi\in \Phi$.}\\
\mathcal{B}(\Phi, \Lambda) &\mbox{multiple Bernoulli series associated to $\Phi$ and $\Lambda$.}\\
H_{\phi} &\mbox{hyperplane in $U$ comprising of vectors $u$ satisfying $\ll u,\phi \rr=0$.}\\
\mathcal{H} &\mbox{arrangement of hyperplanes}\\
\Phi^{eq}(\CH) &\mbox{a set of equations for $\CH$.}\\
\CR_{\mathcal{H}} &\mbox{ring of rational functions on $U$ with poles along $\mathcal{H}$.}\\
\CS_{\mathcal{H}} &\mbox{a subspace of $\CR_{\mathcal{H}}$ given in Definition \ref{def:simple}.}\\
\CG_{\mathcal{H}} &\mbox{a subspace of $\CR_{\mathcal{H}}$ given in Definition \ref{def:simple}.}\\
{\bf R} &\mbox{projector from $\CR_{\mathcal{H}}$ to $\CS_{\mathcal{H}}$.}\\
\mathcal{T}({\mathcal H},\Lambda) &\mbox{topes associated to the system $(\mathcal H,\Lambda)$; $\tau$  a tope.}\\
\mathfrak{B}(\Phi^{eq}) &\mbox{the set of subsets of $\Phi^{eq}$ forming a basis for $V$.}\\
\mathcal{M}_\mathcal{H}, \hat{\CR}_{\mathcal H} &\mbox{spaces of functions defined in \ref{def:mh}.}\\
V_{reg}=V_{reg}(\CH, \Lambda) &\mbox{subset of $V$ regular with
respect to $(\CH,\Lambda)$.}
\end{array}\]

\section{Szenes formula for multiple  Bernoulli  series}\label{sub:hyper}

\subsection{Functions on complement of  hyperplanes}
In this subsection,  $U$ is an $r$-dimensional  complex vector
space, and we recall briefly some structure theorems for the ring of
rational functions that are regular on the complement of a union of
hyperplanes \cite{brver1}.

Let $V$ be the dual vector space to $U$.
If $\phi\in V$, we denote by $H_\phi=\{u\in U;\ll\phi,u\rr=0\}$.

Let $\CH=\{H_1,\ldots, H_N\}$ be  a  set of   hyperplanes in $U$. Then,
we may choose $\phi_k \in V$ such that $H_k=H_{\phi_k}$;  the element $\phi_k$ will be called an equation of $H_k$.
Clearly, an equation $\phi_k$ is not unique,
it is determined up to a nonzero scalar multiple.

We consider  $$U_{\mathcal H}:=\{u\in U; \langle \phi_k,u\rangle\neq
0 \,\, {\rm for \, all}\, k\,\},$$  an open subset of $U$. An
element of $U_{\mathcal H}$ will be called {\it regular}.

\begin{definition}
We denote by $S(V)$ the symmetric algebra of $V$  and identify it with
the ring of polynomial functions on $U$.

We denote by $\CR_{\mathcal H}$ the ring of regular rational
functions on $U$ that are regular on $U_\CH$. That is, the ring
generated by $S(V)$ together with inverses of the linear forms
$\phi_k$ defining $\CH$.

\end{definition}

The ring $\mathcal D(U)$ of differential operators on $U$ with polynomial coefficients acts on $\CR_{\mathcal H}$.
In particular, $U$ operates on $\CR_\mathcal H$ by differentiation.
We denote by $\partial(U)\CR_\CH$ the space of functions in $\CR_\mathcal H$ obtained by differentiation.

If $V$ is one dimensional, then the ring $\CR_\CH$ is the ring of Laurent polynomials $\C[z,z^{-1}]$, and the function $z^i$,
for $i\neq -1$, is obtained as a derivative $\frac{d}{dz}\frac{1 }{i+1} z^{i+1}$.  Thus $\CR_\CH=\frac{d}{dz}\CR_\CH\oplus \C z^{-1}$.
If $f=\sum_{n} a_n z^n$ is an element of $\C[z,z^{-1}]$, we denote by  $\Res_{z=0}f$ the coefficient $a_{-1}$ of $z^{-1}$ in the expression of $f$.
This linear form is characterized by the fact that it vanishes on $\frac{d}{dz} \CR_\CH$.

\medskip

By analogy to the one dimensional case, a linear functional on
$\CR_{\CH}$ vanishing on   $\partial(U)\CR_\CH$  will be called a
`residue'.

\medskip
Let us thus analyze the space   $\CR_{\CH}$ modulo  $\partial(U)\CR_\CH$.
\medskip

Let us consider  a set  $\Phi^{eq}:=\{\phi_1,\phi_2,\ldots, \phi_N\}$  of equations for $\mathcal H$. A subset
$\sigma$ of $\Phi^{eq}$ will be called a \textit{basis} if the elements $\phi_k$  in $\sigma$ form a basis of $V$.
We denote by $\mathfrak{B}(\Phi^{eq})$ the set of such subsets $\sigma$.
A subset  $\nu$ of $\Phi^{eq}$ will be called  \textit{generating} if  the elements $\phi_k$  in $\nu$ generate the vector space $V$.

\begin{definition}\label{def:simple}

$\bullet$ Let $\sigma:=\{\alpha_1,\alpha_2,\ldots,\alpha_r\} \in \mathfrak{B}(\Phi^{eq})$.  Consider the `simple fraction'
$$f_\sigma(z):=\frac{1}{\prod_{k=1}^r \alpha_k(z)}.$$
We denote by $\CS_{\mathcal H}$ the subspace of $\CR_{\mathcal H}$ generated by the elements
$f_\sigma, \; \sigma \in \mathfrak{B}(\Phi^{eq})$.

$\bullet$ Let $\nu=[\alpha_1,\ldots, \alpha_k]$ be  a sequence of $k$ elements of $\Phi^{eq}$ and ${\bf n}=[n_1,n_2,\ldots,n_k]$ be a sequence of positive integers.
We define
$$\theta(\nu,{\bf n})=\frac{1}{\alpha_1^{n_1}\cdots \alpha_{k}^{n_k}}.$$

$\bullet$ We denote by ${\mathcal G}_{\mathcal H}$ the subspace of $\CR_{\mathcal H}$ generated by the elements
$\theta(\nu,\bf{n})$ where $\nu$  is generating.
\end{definition}

As the notation suggests, the spaces $\CR_\CH$, $\CS_\CH$ and  ${\mathcal G}_{\mathcal H}$ depend only on $\CH$.
The term simple fraction comes from the fact that if $\sigma=\{\phi_1,\phi_2,\ldots,\phi_r\}$ is a basis, then we can choose coordinates $z_i$ on $U$ so that $\phi_i(z)=z_i$, so that for this system of coordinates
$f_\sigma(z)=\frac{1}{\prod_{i=1}^r z_i}$.

\medskip

We recall the following `partial fraction' decomposition theorem.
\begin{lemma}\label{ind}
Let $\nu$ be a subset of $\Phi^{eq}$ generating a $t$ dimensional subspace of $V$.
 Then $\theta(\nu,{\bf n})$ may be written as a linear combination of elements
$\theta(\sigma,{\bf m})=\frac{1}{\alpha_{i_1}^{m_1}\cdots \alpha_{i_t}^{m_t}}$ where
$\sigma:=\{\alpha_{i_1}, \ldots,\alpha_{i_t}\}$ is a subset of $\nu$ consisting of $t$ independent elements
and ${\bf m}=\{m_1,\ldots, m_t\}$ a sequence of positive integers.
\end{lemma}

\begin{example}
%If $\Phi=\{z_1,z_2,z_1+z_2\}$, then
$$\frac{1}{z_1z_2(z_1+z_2)}=\frac{1}{z_1(z_1+z_2)^2}+
\frac{1}{z_2(z_1+z_2)^2}.$$
\end{example}\bigskip

Finally, the following  theorem is proved  in Brion-Vergne \cite{brver1}.
\begin{theorem}
$$\CR_{\mathcal H}=\partial(U)\CR_{\mathcal H}\oplus \CS_{\mathcal H}.$$
\end{theorem}

The projector, denoted by ${\bf R}$, from $\CR_{\mathcal H}$ to $\CS_{\mathcal H}$ will be called the
\textit{total residue}.

In view of this theorem, a residue is just a linear form on
$\CS_{\mathcal H}$.  When $\mathcal H$ is the set of hyperplanes
with equations the positive coroots of a simple compact Lie group
$G$, the dimension of $\CS_{\mathcal H}$ is given by the product of
exponents of $G$ (\cite{os}). In Section \ref{section:clroot}, we
will give an explicit basis for $\CS_{\mathcal H}$ for simple Lie
algebras of type $A,B$ and $C$ (which defines the same set of
hyperplanes as $B$) with dual basis consisting of iterated residues.

\subsection{Szenes polynomial}
In this section and for the rest of the article, $V$ will denote a {\bf real} vector space of dimension $r$.

Let $U$ be the dual vector space of $V$. Let $\Lambda$ be a lattice in $V$ with dual lattice $\Gamma$ in $U$.

Let $\CH:=\{H_1,H_2,\ldots, H_N\}$ be a real arrangement of  hyperplanes  in $U$.
We say that $\Lambda$ and $\CH$ are \textit{compatible} if the hyperplanes in $\CH$ are rational with respect to
$\Lambda$, that is, they can be defined by equations $\phi_k\in \Lambda$.  If $\Lambda'$ is another lattice commensurable
with $\Lambda$, then   $\Lambda'$ and $\CH$ are also compatible.

Thus we now consider a lattice $\Lambda$ and a real arrangement of hyperplanes
$\CH=\{H_1,H_2,\ldots, H_N\}$ in $U$ rational
with respect to $\Lambda$.

We choose  $\Phi^{eq}=\{\phi_1,\phi_2,\ldots,\phi_N\}$, a set of defining equations for $\CH$, with each
$\phi_i$ in $\Lambda$. We sometimes refer to $\CH$ only via its set of equations $\Phi^{eq}$ and
write $\CH=\cup\{\phi_k=0\}$.

We denote  the complex arrangement defined by $\cup\{\phi_k=0\}$ in $U_{\C}$ by the same letter $\CH$.
We denote by $U_\mathcal H=\{\prod_k \phi_k\neq 0\}$ the corresponding open subset of $U_\C$.

An \textit{admissible hyperplane} $W$ in $V$  (for the system $\CH$) is an hyperplane generated by $(r-1)$ linearly
independent elements $\phi_k$ of $\Phi^{eq}$. Such an hyperplane will also be called an (admissible) \textit{wall}.
An admissible \textit{affine} wall is a translate of a wall by an element of $\Lambda$.

An element $v\in V$ is called \textit{regular} for $(\CH,\Lambda)$  if $v$ is not on any affine wall (we will just say that $v$ is  regular).  The meaning of the word regular is thus different for elements $v\in V$ ($v$ is not on any affine wall) and $u\in U_\C$
($u$ is such that $\prod_k \langle \phi_k,u\rangle\neq 0$). However, it will be clear what regular means in the context.

A \textit{tope} $\tau$ is a connected component of the complement of all affine hyperplanes. Thus a tope $\tau$ is a
connected open subset of $V$ consisting of regular elements.
We denote the set of topes by $\mathcal{T}(\CH,\Lambda)$. As the notation indicates, $\mathcal{T}(\CH,\Lambda)$ does not depend on the choice of  equations  for $\CH$.

\begin{example}\label{ex:two}
Let $V=\R e_1\oplus \R e_2$  and  $\Lambda=\Z e_1\oplus \Z e_2$.  Let $U$ be its dual with basis $\{e^1,e^2\}$.
We express $z\in U_\C$ as $z=z_1 e^1+z_2 e^2$, and consider  the set of hyperplanes $$\CH=\{\{z_1=0\}, \{z_2=0\}, \{z_1+z_2=0\}\}$$
with the set of equations $\Phi^{eq}=\{e_1,e_2, e_1+e_2\}$.
Figure \ref{a2} depicts topes associated to this pair.

\begin{figure}
\begin{center}
 \includegraphics[width=37mm]{tope2mod.mps}\\
 \caption{$\CT(\CH, \Lambda)$ for Example \ref{ex:two}}\label{a2}
 \end{center}
\end{figure}

\end{example}

\begin{example}\label{ex:btwo}
Let $V=\R e_1\oplus \R e_2$  and  $\Lambda=\Z e_1\oplus \Z e_2$.
Let $U$ be its dual with basis $\{e^1,e^2\}$.
We express $z\in U_\C$ as $z=z_1 e^1+z_2 e^2$, and consider  the set of hyperplanes
$$\CH=\{\{z_1=0\}, \{z_2=0\}, \{z_1+z_2=0\}, \{z_1-z_2=0\}\}$$
with the set of equations $\Phi^{eq}=\{e_1,e_2, e_1+e_2, e_1-e_2\}$.
Figure \ref{b2} depicts topes associated to this pair.
\end{example}

\begin{figure}
\begin{center}
 \includegraphics[width=37mm]{tope3mod.mps}\\
 \caption{$\CT(\CH, \Lambda)$ for Example \ref{ex:btwo}}\label{b2}
 \end{center}
\end{figure}

We denote by $V_{reg}(\CH,\Lambda)$ (or simply $V_{reg}$) the set of ($\CH,\Lambda$) regular elements of $V$.
It is an open subset of $V$ which is the disjoint union of all topes.

A locally constant function on $V_{reg}$ is a function on $V_{reg}$ which is constant on each tope.
A piecewise polynomial function on $V_{reg}$ is a function on $V_{reg}$ which is given by a polynomial expression on each
tope.

If $t\in \R$, we denote by $[t]$ the integral part of $t$, and by $\{t\}=t-[t]$ the fractional part of $t$.
If $\gamma \in \Gamma$ vanishes on an admissible hyperplane $W$, and $c$ is a constant, then the function
$v\to \{\langle\gamma,v\rangle+c\}$ is piecewise polynomial (piecewise linear) and is periodic with respect to $\Lambda$.
Szenes residue formula provides an  algorithm  to describe
 Bernoulli series in terms of these basic functions.

\bigskip

\begin{definition}\label{def:mh} Let $\mathcal M_{\mathcal H}$ be the space of functions $h/Q$ where $Q$ is a product of
linear forms belonging to $\Phi^{eq}$, and $h$  a holomorphic function defined in a neighborhood of $0$ in $U_\C$.

We define the space $\hat \CR_{\mathcal H}$ as the space of functions $\hat h /Q$ where
$\hat h=\sum_{k=0}^{\infty} P_k$ is a  formal power series and $Q$ is a product of linear forms belonging to
$\Phi^{eq}$ as before.
\end{definition}

Taking the Taylor series $\hat h$ of $h$ at $0$ defines an
injective map from  $\mathcal M_{\mathcal H}$ to $\hat \CR_{\mathcal H}$.
The projector ${\bf R}$ from $\CR_{\mathcal H}$ to $\CS_{\mathcal H}$ extends to $\hat \CR_{\mathcal H}$.
Indeed ${\rm {\bf R}}$ vanishes outside the homogeneous components of degree $-r$ of the graded space $\CR_{\mathcal H}$.
Thus if $h/Q$ is an element in $\mathcal M_{\mathcal H}$, with $Q$ a product of $N$ elements of $\Phi^{eq}$, we take the Taylor series $[h]_{N-r}$ of $h$ up to order $N-r$, and define ${\rm {\bf R}}(\frac{h}{Q})={\rm {\bf R}} (\frac{[h]_{N-r}}{Q}).$
For example, the equality
$$\frac{e^{zt}}{e^z-1} = \frac{1}{z} (z\frac{e^{zt} }{e^z-1})$$
identifies  the function  $\frac{e^{zt}}{e^z-1}$ to an element of $\mathcal M_{\mathcal H}$ with ${\mathcal H}=\{0\}$.
Note that each homogeneous term of the Taylor  series expansion
$$z\frac{e^{zt}}{e^z-1}=\sum_{k=0}^{\infty} B(k, t)\frac{z^k}{k!},$$ where $B(k,t)$ is the $k^{\text{th}}$ Bernoulli
polynomial in $t$ as before, is a polynomial in $t$.

Let $f\in \CR_{\CH}$, $z\in U_{\CH}$ and $\gamma\in \Gamma$.
Then if $z$ is small,   $2i\pi\gamma-z$ is still a regular element of $U$. Consider the series
$$S(f,z,v)= \sum_{\gamma\in \Gamma}  f(2i\pi\gamma-z) e^{ \langle v,2i\pi \gamma\rangle }.$$
When $f$ decreases sufficiently quickly at infinity, the series
$S(f,z,v)$ is absolutely convergent and defines a continuous function of $v$.
In general, as $f$ is of at most polynomial growth, the series  $$\sum_{\gamma\in \Gamma} f(2i\pi\gamma-z) e^{ \langle v,2i\pi \gamma\rangle }$$
is the Fourier series of a  generalized function on  $V/\Lambda$.

Multiplying $S(f,z,v)$ by the exponential $e^{-\langle z,v\rangle}$ we introduce the following definition.

\begin{definition}
Let $f\in \CR_{\CH}$, $z\in U_{\CH}$ and small. We define the generalized function of $v$ by
$$\AV^{\Lambda}(f)(z,v)= \sum_{\gamma\in \Gamma}  f(2i\pi\gamma-z) e^{ \langle v,2i\pi \gamma-z\rangle }.$$
\end{definition}

The meaning of $\AV^{\Lambda}(f)$ is clear : average the function  $z\mapsto  f(-z)e^{-\langle v,z \rangle}$ over
$2i\pi \Gamma$ in order to obtain a function   on the complex torus $U_\C/2i\pi \Gamma$.

We  consider  $\AV^{\Lambda}(f)(z,v)$  as a  generalized function of $v\in V$ with coefficients meromorphic functions
of $z$ on $U_\C/2i\pi\Gamma$.  In fact, as we will see,
when $f$ is in $\CS_\CH$, the convergence of the series
$$\sum_{\gamma\in \Gamma}  f(2i\pi\gamma-z) e^{ \la v,2i\pi \gamma \ra }$$ holds in the sense of the Fourier series of an
$L^2$- periodic function of $v\in V/\Lambda$, and
$$v\to e^{-\ll v,z\rr}\left( \sum_{\gamma\in \Gamma}  f(2i\pi\gamma-z) e^{ \langle v,2i\pi \gamma\rr }\right)=\AV^{\Lambda}(f)(z,v)$$
is a locally constant function of  $v\in V_{reg}$ with values in $\CM_{\CH}$.

 %Thus, when $f\in \CS_\CH$, and $v$ varies in a tope $\tau$, the dependance of $v$ of the function
%%$$S(f,z,v)=e^{\langle z,v\rangle } \AV^{\Lambda}(f)(z,v)$$
%%is simply through the factor $e^{\langle z,v\rangle }$.
%%

\bigskip

Note the covariance relation. For $\lambda\in \Lambda$,

\begin{equation}\label{eq:covariance}
\AV^{\Lambda}(f)(z,v+\lambda)=e^{-\ll \lambda,z\rr}\AV^{\Lambda}(f)(z,v).
\end{equation}

It is easy to compare
$\AV^{\Lambda}(f)(z,v)$ when we change the lattice $\Lambda$.

\begin{lemma}\label{lem: first compare}
Let $f\in \CR_\CH$. If $\Lambda^1\subset \Lambda^2$,  then
\begin{equation}\label{eq:SZlattices}
\AV^{\Lambda^2}(f)(z,v)=|\Lambda^2/\Lambda^1|^{-1}
\sum_{\lambda\in \Lambda^2/\Lambda^1}\AV^{\Lambda^1}(f)(z,v+\lambda).
\end{equation}
\end{lemma}

\begin{proof}
Denote the dual of $\Lambda^i$ by $\Gamma^i$.  Then,
$$\AV^{\Lambda^1}(f)(z,v+\lambda)= \sum_{\gamma\in \Gamma^1}  f(2i\pi\gamma-z) e^{ \langle v+\lambda,2i\pi \gamma-z\rangle }$$
and the sum over $\lambda\in \Lambda^2/\Lambda^1$ of
$e^{\ll \lambda,2i\pi \gamma\rr}$ is equal  to $0$ except when $\gamma\in \Gamma^2$.
\end{proof}

\begin{example}\label{ex:dim1Eis}
Let $V=\R$, $\Lambda=\Z$, $z\in \C$ small, $z\neq 0$,  and $f(z)=\frac{1}{z}$.  For $v\in \R$,
$$\AV^{\Lambda}(f)(z,v)=\sum_{n\in \Z}  \frac{e^{ v(2i\pi n-z)}}{(2i\pi n -z)}= e^{-zv}\sum_{n\in \Z}  \frac{e^{ 2i\pi n v}}{(2i\pi n -z)}.$$
This series is not absolutely convergent, but the  oscillatory factor $e^{2i\pi nv}$ insures the convergence in the distributional sense as a function of $v$.
We have
%We can take the derivative with respect to $v$, and obtain
% $$\partial_v (\AV^{\Lambda}(f)(z,v))=\sum_{n\in \Z}  e^{ v(2i\pi n-z)}=e^{-vz}\delta_\Z(v).$$
%Thus, for $v$ in each open interval $]n,n+1[$, the derivative is $0$, so that the function is constant. Passing through an integer $n$, the constant value changes by
%$e^{-nz}$. More precisely, in this one  dimensional case, we obtain
\begin{equation}\label{eq:AV}
\AV^{\Lambda}(f)(z,v)=\frac{e^{-[v]z}}{1-e^z}
\end{equation}
(recall that  $[v]$ denotes the integral part of $v$).

Indeed, let us compute the $L^2$-expansion of the periodic function
$v\mapsto \frac{e^{(v-[v])z}}{1-e^z}$.
By definition,  this  is
$$\sum_{n\in \Z} \left(\int_{0}^1 \frac{e^{(v-[v])z}}{1-e^z} e^{-2i\pi n v} dv\right) e^{2i\pi n v}=
\sum_{n\in \Z} \left(\int_{0}^1 \frac{e^{v(z-2i\pi n)}}{1-e^z} dv\right) e^{2i\pi n v}
$$
$$=\sum_{n\in \Z}  \frac{e^{ (z-2i\pi n)}-1}{(1-e^z)(z-2i\pi n)} e^{2i\pi n v}=\sum_{n\in \Z}  \frac{1}{(2i\pi n -z)} e^{2i\pi n v}.$$

\end{example}

We see in this one dimensional example  that $\AV^{\Lambda}(f)(z,v)$ is a locally constant function of $v$.

%\begin{definition}
%Let $v$ be in a tope $\tau$.
%Let $f\in \CS_{\CH}$, $z\in U_\mathcal H$ small, we define
%$$Eis(\Lambda,\tau)(f)(z)= \sum_{\gamma\in \Gamma}  f(2i\pi\gamma-z) e^{ \langle v,2i\pi \gamma-z\rangle }$$
%where $v$ is any point of $\tau$.
%\end{definition}
%
%

In general, we have the following proposition.
\begin{proposition}
If $f\in \CS_\CH$, the function
$v\to \AV^{\Lambda}(f)(z,v)$ is a locally constant function on $V_{reg}$, with values in $\CM_\CH$.
\end{proposition}

We prove this by computing
$\AV^{\Lambda}(f)(z,v)$ explicitly for a simple fraction $f=f_\sigma$.
Recall that the set of equations $\Phi^{eq}$ is a subset  of  $\Lambda$.
Let $\sigma=\{\alpha_1,\alpha_2,\ldots,\alpha_r\}$ be an element of $\mathfrak{B}(\Phi^{eq})$.
The elements $\alpha_k$ belong to $\Lambda$. Let $Q_\sigma:=\oplus_{k=1}^r[0,1)\alpha_k$  be the semi-open parallelepiped spanned by $\sigma$.

\begin{definition}
Let $v$ be regular in $V$, and let $\sigma\in \mathfrak B(\Phi^{eq})$ be a basis. Define
$\Trans(v,\sigma)$  to be the set of elements $\lambda\in \Lambda$ such that $v+\lambda\in Q_\sigma$.
\end{definition}

This set depends only on the tope $\tau$ where $v$ belongs, hence we denote it by $\Trans(\tau,\sigma)$.

Let $\Lambda_\sigma$ be the sublattice  of $\Lambda$ generated by the elements in the basis $\sigma$. Then
the set $\Trans(\tau,\sigma)$ contains exactly $\Lambda/\Lambda_\sigma$ elements.

\begin{proposition}\label{pro:formulaEIS}
If $v\in \tau$ and $\sigma=\{\alpha_1,\alpha_2,\ldots,\alpha_r\} \in \mathfrak{B}(\Phi^{eq})$,
$$\AV^{\Lambda}(f_\sigma)(z,v)=\frac{1}{|\Lambda/\Lambda_{\sigma}|} \sum_{\lambda\in \Trans(\tau,\sigma)} \frac{e^{\langle \lambda,z\rangle }}{\prod_{i=1}^r (1-e^{\langle \alpha_i,z\rangle})}.$$
\end{proposition}

\begin{proof}
If $\Lambda=\Lambda_\sigma$, the formula reduces to the one  dimensional case. Otherwise, we use Lemma
\ref{lem: first compare} and the covariance formula \ref{eq:covariance}.
\end{proof}

The dependance of $\AV^{\Lambda}(f_\sigma)(z,v)$ on $v$ is only via the tope $\tau$ where $v$ belongs.
Thus we see that, for any $f\in \CS_\CH$, the function $v\mapsto \AV^{\Lambda}(f)(z,v)$ is  a locally constant function on $V_{reg}$ with value in $\CM_\CH$.

\begin{example}
We return to the Example \ref{ex:two}, where $\Phi^{eq}=\{e_1,e_2,e_1+e_2\}$.
To describe the function $v\mapsto \AV^{\Lambda}(f)(z,v)$ on $V_{reg}$ completely, it suffices
to give its expression on each tope $\tau_1$ and $\tau_2$ which are depicted in Figure \ref{a2}.  This is true since any
element of $V_{reg}$ can be translated to $\tau_1$ or $\tau_2$ by an element of $\Lambda$, and then one can use the
covariance relation (\ref{eq:covariance}).

Choose $\sigma=\{e_1,e_1+e_2\}$ a basis of $\Phi^{eq}$.
Write $z=z_1e^1+z_2 e^2$ in the dual space, then
$f_\sigma(z)=\frac{1}{z_1(z_1+z_2)}$ is in $S_{\CH}$.

For   $v\in \tau_1$
$$\AV^{\Lambda}(f_\sigma)(z,v)=\frac{e^{z_1}}{(1-e^{z_1})(1-e^{z_1+z_2})},$$
while if $v\in \tau_2$
$$\AV^{\Lambda}(f_\sigma)(z,v)=\frac{1}{(1-e^{z_1})(1-e^{z_1+z_2})}.$$
\end{example}

\bigskip
\bigskip
%\begin{example}
%Consider  $\Phi^{eq}=\{e^1,e^2, e^1+e^2, e^1-e^2\}$ and  $\Lambda=\Z e^1\oplus \Z e^2$.
%Consider the tope $\tau=\{v_1e^1+v_2 e^2; v_1>v_2>0, v_1+v_2<1\}$.
%
%Choose $\sigma=\{e^1-e^2,e^1+e^2\}$, then for $z=z_1e_1+z_2 e_2$ in the dual space, and $v\in\tau$
%$$\AV^{\Lambda}(f_\sigma)(z,v)=\frac{1+e^{z_1}}{(1-e^{z_1-z_2})(1-e^{z_1+z_2})}.$$
%\end{example}
%

For $v\in V_{reg}$, denote by
$\SZ^{\Lambda}(v): \CS_{\mathcal H}\to \mathcal M_{\mathcal H}$
the map
 $$(\SZ^{\Lambda}(v)f)(z)=\AV^{\Lambda}(f)(z,v).$$

This operator is locally constant.
We denote its value on $\tau$ by $\SZ^{\Lambda}(\tau)$:
$$(\SZ^{\Lambda}(\tau)f)(z)=\AV^{\Lambda}(f)(z,v)$$
for any choice of $v\in \tau$.

\bigskip

We now define a  piecewise polynomial function of $v$ associated to a  function $g(z)$ in $\CR_\CH$.

First, the operator  on  $\CM_\CH$   given by  multiplication by a function $h(z)$,  that is $f(z)\mapsto h(z)f(z)$,  is
 simply denoted  by $f\mapsto hf$.

If $v\in V$ and  $g\in \CR_{\mathcal H}$, then $g_v(z)=g(z)e^{\ll z,v\rr}$ is a function in $\CM_\CH$ depending on $v$.

Let $v\in V_{reg}$. Consider the map $\CS_\CH\to \CM_\CH$ which associates to
$f\in \CS_\CH$ the function $g(z)e^{\ll z,v\rr}(\SZ^{\Lambda}(v)f)(z)$. We project back this function on $\CS_\CH$ using the projector
${\bf R}$. Thus the map
\begin{equation}\label{Sz}
f(z)\mapsto  {\bf R}\left(g(z)e^{\ll z,v\rr}
(\SZ^{\Lambda}(v)f)(z)\right)
\end{equation}
is a map from $\CS_\CH$ to $\CS_\CH$ depending on $v$.
As $\CS_\CH$ is finite dimensional, we can take the trace of this operator, and thus obtain a function of $v\in V_{reg}$.
Let us record this definition.

\begin{definition}\label{def:sz}
Let $g\in \CR_{\CH}$.
Define  the function $P(\CH,\Lambda,g)$  on $V_{reg}(\CH,\Lambda)$ by $$P(\CH,\Lambda,g)(v):=Tr_{\CS_\CH} \left({\bf R}\, g_v \,\SZ^{\Lambda}(v)\right).$$
\end{definition}
Let us see that $P(\CH,\Lambda,g)(v)$ is a polynomial function of $v$ on each tope $\tau$.
Indeed, to compute $P(\CH,\Lambda,g)(v)$ using (\ref{Sz}), we have to compute the total residue of  functions
$g(z)e^{\ll z,v\rr}\AV^{\Lambda}(f_i)(z,v)$ with $f_i$ varying over a basis of $\CS_\CH$.
If $v\in \tau$, then  $\AV^{\Lambda}(f_i)(z,v)=\SZ^{\Lambda}(\tau)f_i(z)$  is constant in $v$. So when $v$ stays in a
tope $\tau$,  the dependance of $g(z)e^{\ll z,v\rr}\AV^{\Lambda}(f_i)(z,v)=g(z)e^{\ll z,v\rr}(\SZ^{\Lambda}(\tau)f_i)(z)$
on $v$ is via  $e^{\ll z,v\rr}$, and the map $\bf{R}$ involves only the Taylor series of this function  up to some order.

\bigskip

Thus we have associated to $g\in \CR_\CH$ (and $\Lambda$) a piecewise polynomial function
$P(\CH,\Lambda,g)$   on $V_{reg}.$

%
%\begin{definition}\label{def:localpoly}
%Let $g\in \CR_\CH$.
%We denote by
%$P(\CH,\Lambda,g,\tau)(v)$
%the
%polynomial function on $V$  such that
%$$P(\CH,\Lambda,g)(v)=P(\CH,\Lambda,g,\tau)(v)$$
%for  $v\in \tau$.
%\end{definition}
%
%%

It is easy to compare piecewise polynomial functions
$P(\CH,\Lambda,g)$   associated to different lattices.

Let $\Lambda^1\subset \Lambda^2$,  then   $V_{reg}(\CH,\Lambda^2)\subset V_{reg}(\CH,\Lambda^1)$.

\begin{lemma}\label{lem compareP }
If $\Lambda^1\subset \Lambda^2$,  then
\begin{equation}\label{eq:SZlattices}
P(\mathcal H, \Lambda^2,g)(v)=|\Lambda^2/\Lambda^1|^{-1}
\sum_{\lambda\in \Lambda^2/\Lambda^1}P(\mathcal H, \Lambda^1,g)(v+\lambda).
\end{equation}
\end{lemma}
This follows immediately from Lemma \ref{lem: first compare}.

\bigskip

Our next aim  is to  compute the piecewise polynomial function
$P(\CH,\Lambda,g)$ using residues. We need more definitions.

An ordered basis of $\Phi^{eq}$ is a sequence
$[\alpha_1,\alpha_2,\ldots,\alpha_r]$ of elements of $\Phi^{eq}$ such that the underlying set is in
$\mathfrak{B}(\Phi^{eq})$. We denote the set of ordered bases by ${\overrightarrow{\mathfrak{B}}}(\Phi^{eq})$.

Let  ${\overrightarrow{\sigma}}=[\alpha_1,\alpha_2,\ldots,\alpha_r] \in {\overrightarrow{\mathfrak{B}}}(\Phi^{eq})$.
Then, to this data, one associates an iterated residue
functional $\Res^{{\overrightarrow{\sigma}}}$ on $\CR_{\mathcal H}$  as follows. For $z\in U$,
let $z_j=\ll z,\alpha_j\rr$. Then a function $f$ in $\CR_{\mathcal H}$ can be expressed as a function
$f(z_1,z_2,\ldots, z_r).$
We define
$$\Res^{{\overrightarrow{\sigma}}}(f):=\Res_{z_1=0}(\Res_{z_2=0}\cdots(\Res_{z_r=0}f(z_1,z_2,\ldots,z_r))\cdots).$$
Clearly $\Res^{{\overrightarrow{\sigma}}}(f_\sigma)=1.$

The functional $\Res^{{\overrightarrow{\sigma}}}$ factors through the canonical projection ${\bf R}:\CR_\mathcal H \to
\CS_\mathcal H$:
$\Res^{{\overrightarrow{\sigma}}}=\Res^{{\overrightarrow{\sigma}}} {\bf R}$.

\begin{definition}
A {\it diagonal subset of} ${\overrightarrow{\mathfrak{B}}}(\Phi^{eq})$  is a subset $\overrightarrow{\mathcal  D}$ of
${\overrightarrow{\mathfrak{B}}}(\Phi^{eq})$
such that the set of simple fractions  $f_\sigma$, ${\overrightarrow{\sigma}}\in \overrightarrow{\mathcal  D}$, forms a basis
of $\CS_{\mathcal H}$: $$\CS_{\mathcal H}=\oplus_{{\overrightarrow{\sigma}}\in \overrightarrow{\mathcal  D}} \C f_\sigma$$
and the dual basis to the basis $\{f_\sigma, {\overrightarrow{\sigma}} \in \overrightarrow{\mathcal  D}\}$ of $\CS_{\mathcal H}$ is the set of
linear forms $\Res^{{\overrightarrow{\sigma}}}$, that is, $\Res^{\overrightarrow{\tau}}(f_\sigma)=\delta_\sigma^\tau$,
for ${\overrightarrow{\sigma}},\overrightarrow{\tau} \in \overrightarrow{\mathcal  D}$.
\end{definition}

A total order  on $\Phi^{eq}$ allows us to construct the set  of Orlik-Solomon bases (see \cite{brver1}), which provides
diagonal basis of $\CS_{\CH}$. However we will also use some other diagonal subsets.

If $B:\CS_{\mathcal H}\to \mathcal
\CM_{\mathcal H}$ is an operator, the trace of the operator $A:={\bf R}B$ is thus $$Tr(A):=\sum_{{\overrightarrow{\sigma}}\in \overrightarrow{\mathcal  D}} \Res^{{\overrightarrow{\sigma}} }B f_\sigma.$$

\begin{definition}\label{def:localpoly}
Let $g\in \CR_\CH$ and $\tau$ a connected component of $V_{reg}$.
We denote by
$P(\CH,\Lambda,g,\tau)(v)$
the
polynomial function on $V$  such that
$$P(\CH,\Lambda,g)(v)=P(\CH,\Lambda,g,\tau)(v)$$
for  $v\in \tau$.
\end{definition}
Hence, we may give a more explicit formula for the polynomial $P(\CH,\Lambda,g,\tau)(v)$ using a set $\overrightarrow{\mathcal  D}$.

\begin{proposition}\label{prop:iteres}
Let $g\in \CM_{\CH}$. Let $\tau\in \CT(\CH,\Lambda)$
be a tope.
Let $\overrightarrow{\mathcal  D}$ be a diagonal subset of ${\overrightarrow{\CB}}(\Phi^{eq})$.
Then
$$P(\CH,\Lambda,g,\tau)(v)=\sum_{\overrightarrow{\sigma} \in \overrightarrow{\mathcal  D}} \Res^{{\overrightarrow{\sigma}}} \Big( e^{\ll z,v\rr}g(z)  \SZ^{\Lambda}(\tau)(f_\sigma)(z)\Big).$$

\end{proposition}

Furthermore
$\SZ^{\Lambda}(\tau)(f_\sigma)(z)$  is given explicitly by
Proposition \ref{pro:formulaEIS}.
Thus, in principle, the above formula allows us to compute
$P(\CH,\Lambda,g)$.

\medskip

It is important to remark that the determination of a diagonal subset   $\overrightarrow{\mathcal  D}$ depends essentially
only on the system of hyperplanes $\CH$ and not on the choice of $\Phi^{eq}$.
The  difficulties in writing an algorithm for
$P(\CH,\Lambda,g)$ lies in the description of a diagonal subset $\overrightarrow{\mathcal  D}$, and also for each
$\sigma\in \overrightarrow{\mathcal  D}$, in the computation of $A^{\Lambda}(f_\sigma)$.
The difficulty of this  last computation depends on the lattice $\Lambda$.

\begin{definition}
Let $\Phi^{eq}\subset \Lambda$.
A basis $\sigma\in \mathfrak B(\Phi^{eq})$ is called unimodular (with respect to $\Lambda$) if  $\sigma$ is a basis of
the lattice $\Lambda$. A set $\Phi^{eq}$ is called unimodular,
if any basis $\sigma\in \mathfrak B(\Phi^{eq})$ is unimodular.
\end{definition}

\begin{example}
The set $$\Phi^{eq}=\{ e_1, e_2,(e_1+e_2), (e_1-e_2)\}$$  is contained in  $\Lambda=\Z e_1+\Z e_2$.
Then $\sigma=\{e_1+e_2,e_1-e_2\}$ belongs to $\mathfrak B(\Phi^{eq})$, and the index of $\Lambda_\sigma$ in
$\Lambda$ is $2$. So $\Phi^{eq}$ is not unimodular.
\end{example}

\begin{definition}
Let $\sigma=\{\alpha_1,\alpha_2,\ldots,\alpha_r\}$ be a  basis.
For $1\leq i \leq r$, the linear form    $v\to c_i^{\sigma}(v)$
is the coefficient of $v$ with respect to $\alpha_i$.
\end{definition}

Consider the function  $\{t\}=t-[t]$.
On each open interval $\tau=]n,n+1[$, the function $\{t\}$ coincides with  the linear function  $t\mapsto (t-n)$.

If $\sigma$ is a unimodular basis, we express $v= \sum_{i=1}^r c_i^{\sigma}(v)\alpha_i$.
Then $$v-\sum_{i=1}^r [c_i^{\sigma}(v)]\alpha_i=\sum_{i=1}^r\{c_i^{\sigma}(v)\}\alpha_i$$ is in $Q_\sigma$.
Thus the set $\Trans(\tau,\sigma)$ contains exactly the element $\lambda=-\sum_{i=1}^r [c_i^{\sigma}(v)]\alpha_i$ (which depends only on the tope $\tau$ where $v$ lies).

\begin{corollary}\label{cor:Step}
Let $\sigma$ be a unimodular basis in $\mathfrak B(\Phi^{eq})$. Let $v\in \tau$, and $\lambda=-\sum_{i=1}^r [c_i^{\sigma}(v)]\alpha_i$.  Then
$$\AV^{\Lambda}(f_\sigma)(z,v)= \frac{e^{\langle \lambda,z\rangle }}{\prod_{i=1}^r (1-e^{\langle \alpha_i,z\rangle})}.$$
\end{corollary}

\medskip

It may happen that even when the system $\Phi^{eq}$ is not unimodular for the lattice $\Lambda$, we can choose $\overrightarrow{\mathcal  D}$ to consist of unimodular bases.
In particular, using Proposition \ref{prop:iteres}, we can give an explicit algorithm for computing the piecewise polynomial function
$P(\CH,\Lambda,g)$ for classical root systems in the form of a step polynomial. Let us define what this means.

\begin{definition}
Let $\overrightarrow{\mathcal  D}$ be a subset of ${\overrightarrow{\CB}}(\Phi^{eq})$.
We denote by $Step(\overrightarrow{\mathcal  D})$ the algebra of functions on $V$ generated by the piecewise linear functions $v\to \{c_i^{\sigma}(v)\}$ with $\sigma$ running over
$\overrightarrow{\mathcal  D}$ and $1\leq i\leq r$.
An element of the algebra $Step(\overrightarrow{\mathcal  D})$ will be called a step polynomial (associated to $\overrightarrow{\mathcal  D}$).

\end{definition}

It is clear that a step polynomial is a periodic function on $V$,
which is expressed by a polynomial formula on each tope.

\begin{proposition}
Let $g\in {\mathcal G}_{\CH}$.
Assume that  $\overrightarrow{\mathcal  D}$ is a diagonal subset of ${\overrightarrow{\mathfrak B}}(\Phi^{eq})$ consisting of unimodular basis (with respect to  $\Lambda$).
Then  the piecewise polynomial function $P(\CH,\Lambda,g)$ belongs  to the algebra $Step(\overrightarrow{\mathcal  D})$.
\end{proposition}

\begin{proof}
This is clear, as we have the formula
\begin{equation}\label{steppoly}
P(\CH,\Lambda,g)(v)=\sum_{\overrightarrow{\sigma} \in \overrightarrow{\mathcal  D}} \Res^{{\overrightarrow{\sigma}}} g(z) e^{\sum_{i=1}^r\{c_i^{\sigma}(v)\}\ll\alpha_i,z\rr}\frac{1}{\prod_{i=1}^r (1-e^{\langle \alpha_i,z\rangle})},
\end{equation}
and the dependance on $v$ is through the Taylor expansion (in $z$) of $e^{\sum_{i=1}^r\{c_i^{\sigma}(v)\}\ll\alpha_i,z\rr}$ up to some order.
\end{proof}

\subsection{Multiple  Bernoulli series  }\label{section:mbs}

We return to our main object of study: the multiple Bernoulli series.

Let $V$, $\Lambda$ and $\CH$ be as before.
We denote by $\Gamma\subset U$ the dual lattice  to $\Lambda$, and by
$\Gamma_{reg}(\CH)$ the set $\Gamma \cap U_{\CH}$.
If $\gamma \in \Gamma_{reg}(\CH)$, a function $g$ in $\CR_\CH$ is defined on $2i\pi \gamma$.

\begin{definition}
If $g\in \CR_{\mathcal H}$,  the generalized function  $\mathcal{B}(\mathcal H, \Lambda,g)(v)$ on $V$  is defined by
$$\mathcal{B}(\mathcal H, \Lambda,g)(v)=\sum_{\gamma \in \Gamma_{\reg}(\mathcal H)} g(2i\pi \gamma) e^{ 2i\pi \ll v ,\gamma \rr}.$$
\end{definition}

The above series converges in the space of generalized functions on $V$.

We state some obvious properties of
$\mathcal{B}(\mathcal H, \Lambda,g)$.

\begin{lemma}\label{lem:secondcompare}
If $\Lambda^1\subset \Lambda^2$, then
\begin{equation}\label{eq:lattices}
\mathcal{B}(\mathcal H, \Lambda^2,g)(v)=|\Lambda^2/\Lambda^1|^{-1}
\sum_{\lambda\in \Lambda^2/\Lambda^1}\mathcal{B}(\mathcal H, \Lambda^1,g)(v+\lambda).
\end{equation}
\end{lemma}

If we dilate a lattice $\Lambda$ by $\ell$, and if $g$ is homogeneous of degree $h$,
we clearly have
\begin{equation}\label{eq:dillattices}
\mathcal{B}(\mathcal H, \ell\Lambda
,g)(v)= \ell^{-h}
\mathcal{B}(\mathcal H, \Lambda,g)(\frac{v}{\ell}).
\end{equation}

With these two properties, we can compare
$\mathcal{B}(\mathcal H, \Lambda,g)$ over commensurable lattices.

\bigskip

\begin{definition}
A generalized function $b$  on $V$ will be called {\it piecewise
polynomial relative to $\CH$ and $\Lambda$} if

$\bullet$ the function $b$ is locally $L^1$,

$\bullet$ for each tope $\tau$ in $\mathcal{T}(\CH,\Lambda)$, there exists a polynomial
function $b^{\tau}$ on $V$ such that the restriction of $b$ to $\tau$
coincides with the restriction of  the polynomial $b^{\tau}$ to
$\tau$.
\end{definition}

As an $L^1$-function is entirely determined by its restriction to $V_{reg}$,  we will not distinguish between  piecewise
polynomial generalized functions on $V$ and piecewise polynomial functions on $V_{reg}$ as defined in the preceding section.

%It is not clear to us what is the `best description'  of a piecewise polynomial function.
%Indeed, we can either describe $b$  as a family of polynomial functions  $b(\tau)$ indexed by topes, or
%we can compute $b$ as a polynomial expression with respect to some basic piecewise polynomials.

{\bf Be careful:}
the restriction to any tope $\tau$ of a piecewise polynomial generalized function  $b$ is  polynomial.
However, the converse is not true. For example the $\delta$ function of the lattice $\Lambda$ restricts to $0$ on
any tope $\tau$, but is not a piecewise polynomial generalized function, as it is not a locally $L^1$-function.

\bigskip

Recall the definition of ${\mathcal G}_\CH$ as given in Definition \ref{def:simple}.
If we multiply $g$ by a polynomial $p$, the function
$v\mapsto \mathcal{B}(\mathcal H, \Lambda,pg)(v)$ is obtained from  the function $\mathcal{B}(\mathcal H, \Lambda,g)(v)$ by differentiation (in the distribution sense).
Any function $f$ in $\CR_\CH$ is of the form $pg$, with $g\in
{\mathcal G}_\CH$. Thus we can reduce the computation of
$\mathcal{B}(\mathcal H, \Lambda,f)$ to the computation of
$\mathcal{B}(\mathcal H, \Lambda,g)$ for $g\in \mathcal G_\CH$.
Thus the following proposition follows  from calculations in dimension one,  Lemma \ref{ind} and the comparison formulae
on different lattices as given in Lemma \ref{lem:secondcompare}.

\begin{proposition}\label{prop:polyn}
If $f\in \CR_{\CH}$, the restriction to any tope $\tau$ of
$\mathcal{B}(\mathcal H,\Lambda,f)$ is given by a polynomial function.

Furthermore, if $f\in {\mathcal G}_\CH$,  the generalized function   $\mathcal{B}(\mathcal H,\Lambda,f)$ is a piecewise polynomial generalized function.
\end{proposition}

Let us emphasize on the subtle difference between  the conditions  $f\in \CR_{\CH}$, or $f\in {\mathcal G}_\CH$.
Consider $f=1$ in the one dimensional space $U$ and
$\mathcal H=\{0\}$.  The function $f$ is not in ${\mathcal G}_{\CH}$.
Let $v\in V$.
Then  $$\mathcal{B}(\mathcal H,\Lambda,f)(v)=\sum_{n\neq 0}e^{2i\pi nv}=-1+\sum_{n\in \Z}e^{2i\pi nv}.$$
Thus  $\mathcal{B}(\mathcal H,\Lambda,f)(v)$ is the constant function  equal to $-1$ on any tope. However,
it has some singular part $\delta_{\Z}$, and is not locally $L^1$.

In contrast, if $f=\displaystyle \frac{1}{z}$, the generalized function
$$\mathcal{B}(\mathcal H,\Lambda,f)(v)=\sum_{n\neq 0}\frac{e^{2i\pi nv}}{2i\pi n}$$ is locally $L^1$ and  equal to the piecewise polynomial
function $-B(1,\{v\})=1/2-\{v\}$ (see Figure \ref{berno}).

\begin{figure}
\begin{center}
  \includegraphics[width=93mm]{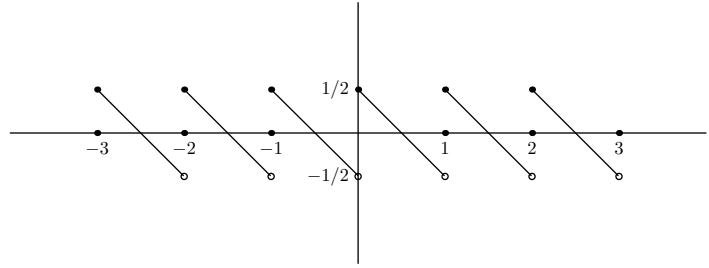}\\
  \caption{Graph of $\mathcal{B}(\{0\},\Z, 1/z)(v)= \frac{1}{2}-\{v\}$}\label{berno}
\end{center}
\end{figure}

\begin{definition}
Let $f\in \CR_{\CH}$.
Given a tope $\tau$ in $\mathcal T(\mathcal H,\Lambda)$, we denote by
$\mathcal{B}(\mathcal H,\Lambda,f,\tau)$
the polynomial function on $V$ which coincides with
$\mathcal{B}(\mathcal H,\Lambda,f)$ on the tope $\tau$.
\end{definition}

\begin{remark}
It is interesting to understand the space of polynomials generated by the polynomial functions
$b^\tau=\mathcal{B}(\mathcal H,\Lambda,f,\tau)$, when $\tau$ runs over the topes, and the wall crossing formula between $b^{\tau_1}$ and $b^{\tau_2}$ when
$\tau_1$ and $\tau_2$ are adjacent.  We addressed some aspects of these theoretical questions in \cite{bover2}.
\end{remark}

Consider  the piecewise polynomial function
$P(\CH,\Lambda,f)$ on $V_{reg}(\CH,\Lambda)$ as given in Definition \ref{def:sz}.

\begin{theorem}(Szenes)\label{theo:main}
Let $f\in \CR_{\CH}$.
On $V_{\reg}(\CH,\Lambda)$, we have the equality
 $$\CB(\CH,\Lambda,f)= P(\CH,\Lambda,f).$$
\end{theorem}

%Thus for $v\in \tau$ the value of the Bernoulli series  at $v$ is  $P(\CH,\Lambda,f,\tau)(v)$.

For completeness, we give a proof of this theorem in the Appendix.

\bigskip

Our Maple program  computes, given data $\mathcal H,\, \Lambda, \, f$,  where $\CH$ is the hyperplane arrangement associated to a classical root system, a piecewise polynomial
function on $V$  in terms of step polynomials. Naturally, we can also evaluate this function at any point $v\in V_{reg}$.

\bigskip

We return to the definition of multiple Bernoulli series in the way  we introduced them in the introduction.

Let $V$ be a vector space with a lattice $\Lambda$ with dual lattice $\Gamma$.
We  considered  $\Phi$  in the introduction as a {\bf list} of elements of $\Lambda$.
We then introduced the following definition. Let
$\Gamma_{\reg}(\Phi)=\{\gamma\in \Gamma ; \ll \phi,\gamma \rr \neq
0\;\mbox{for all} \;\phi\in \Phi\}$ and defined

\begin{equation}
\mathcal{B}(\Phi, \Lambda)(v)=\sum_{\gamma \in \Gamma_{\reg}(\Phi)}  \frac{e^{\ll 2i\pi v ,\gamma \rr}}{
\prod_{\phi \in \Phi} \ll 2i\pi \phi,\gamma \rr}. \end{equation}

Consider $\CH=\cup_k\{\phi_k=0\}$
(some elements of the list $\Phi$ might define the same hyperplane)
and $g(z)=1/\prod_{\phi \in \Phi} \ll  \phi,z \rr$, then
$$\mathcal{B}(\Phi, \Lambda)(v)= \mathcal{B}(\mathcal H, \Lambda,g)(v).$$
We will also call the functions $\mathcal{B}(\mathcal H, \Lambda,g)(v)$
multiple Bernoulli series.

\medskip

\section{Classical root systems}\label{section:clroot}

Let $G$ be a simple, simply connected, compact Lie group of rank $r$
with maximal torus $T$.  We denote the Lie algebra of $T$ and $G$ by
$\frt$ and $\frg$ respectively. Then the complexification $\frh:=\frt_{\C}$ is a
Cartan subalgebra of $\frg_{\C}$.

For $\alpha \in \frh^*$, define
$(\frg_{\C})_{\alpha}=\{X \in \frg_{\C}; [H,X]=\ll \alpha,H\rr X\;\text{for all}\;H \in\frh\}$.
If $\alpha \neq 0$ and
$(\frg_{\C})_{\alpha}\neq 0$, then $\alpha$ is called a {\it root} of
$\frh$ in $\frg_{\C}$. Let $R=R(\frg_{\C},\frh)\subset \frh^*$ be
the set of roots; roots $\alpha \in R$ take imaginary values on $\frt$.
We denote the root lattice by $Q$ and its dual, the coweight lattice, by $\check{P}$.

For  $\alpha \in R$, there exists a unique element $H_{\alpha}$ in
$[(\frg_{\C})_{\alpha},(\frg_{\C})_{-\alpha}]$ satisfying $\ll
\alpha,H_\alpha \rr=2$; it is called the {\it coroot} associated to
the root $\alpha$.  For any $\alpha \in R$, $i H_{\alpha}$ is in $\frt$, and
for any $\alpha, \beta \in R$, $\beta(H_\alpha)$ is integral.
The lattice spanned by $H_{\alpha}$ is
called the coroot lattice and denoted by $\check{Q}$.

Define the weight lattice $P=\{ \lambda \in \frh^* ; \lambda(H_{\alpha}) \in
\mathbb{Z},\; \forall \alpha \in R\}$; it is the dual of the coroot lattice $\check{Q}$.
A {\it regular} weight $\lambda\in P^{reg}$ is such that  $ \lambda(H_{\alpha}) \neq 0$
for all $H_\alpha$.

We denote by
$\frh_\R:=\sum_\alpha \R H_\alpha$, the real span of coroots.  In this section we have $V=\mathfrak h_\R$, and
its dual $\frh_\R^*$ is denoted by $U$ as before. We follow the notation of Bourbaki for root data.

\subsection{Diagonal subsets}
To compute  multiple Bernoulli series associated to classical root systems we need to construct
explicit diagonal bases for the corresponding $\CS_{\mathcal H}$.
Such bases can be constructed by an algorithmic procedure, based on Orlik-Solomon construction.
However in some cases one can describe a diagonal subset $\overrightarrow{\mathcal  D}$ of ${\overrightarrow{\CB}}(\Phi^{eq})$
whose associated simple fractions form a basis for $\CS_{\mathcal H}$ in a direct way, and that we present now.

\subsubsection{The system  of type $A_r$}\label{A}

Let $n=r+1$.
We consider $\R^{n}$ with standard basis $\{e_1,e_2,\ldots,e_n\}$.
Let $$A_r:=[(e_i-e_j);\, 1\leq i< j\leq n]$$ be the root system of type $A$ and rank $r$.

Let $\{e^i\}$ be the dual basis to $\{e_i\}$ and
$$V=\{v=\sum_{i=1}^n v_i e^i;\, \sum_{i=1}^n v_i=0\}.$$

Let $z=\sum_{i=1}^n z^i e_i$ be in $U$ (hence
$\sum_{i=1}^n z^i=0$)
and let $\CH_r^A$ be the system of hyperplanes in $U$ given by $$\CH_r^A=\cup_{1\leq i<j\leq n}\{z^i- z^j=0\}.$$
We take the set
$$\Phi^{eq}(A_r)=\{e^i-e^j; 1\leq i<j\leq n\}$$ of positive coroots as equations of $\CH_r^A$.

\medskip

One way to find a  diagonal basis of $\CS_{\CH_r^A}$ is as follows.

Let $\Sigma=[e^1-e^2,e^2-e^3,\ldots,e^{r}-e^{r+1}]$ be the set of simple coroots.
For a permutation $w$, we denote by $\overrightarrow{\sigma}_w=[e^{w(i)}-e^{w(i+1)},\, i=1,\ldots,r]$.
Then $\overrightarrow{\sigma}_w$ is an ordered basis associated to $w$, and the corresponding
simple fraction is $f_w(z):=\frac{1}{\prod_{i=1}^r (z^{w(i)}-z^{w(i+1)})}$.

Let $W_r$ be the subset of the Weyl group   $\Sigma_{r+1}$ of permutations of $\{e^1,e^2,\ldots, e^{r+1}\}$  leaving the
last element $e^{r+1}=e^n$ fixed.
Recall the following result (see for example Baldoni-Vergne \cite{balver} for a proof).

\begin{proposition}\label{pro:OW}
The set $\overrightarrow{\mathcal D}_W$  consisting of  ordered bases $\overrightarrow{\sigma}_w$  for
$w\in W_r$ is a diagonal subset of ${\overrightarrow {\mathfrak B}}(\Phi^{eq}(A_r))$.
\end{proposition}

We use the above basis in our Maple program.

\medskip
We now give another interesting diagonal subset.

\medskip

Consider  a sequence $ \overrightarrow{\sigma}=[\alpha_2,\alpha_3,\ldots,\alpha_n]$ where
$\alpha_i=e^i-e^j$ with $j<i$.
That is, $\alpha_2=e^2-e^1$, $\alpha_3=e^3-e^2$ or $e^3-e^1$, $\alpha_4=e^4-e^3$, or $e^4-e^2$, or $e^4-e^1$, etc.
Clearly,  $\overrightarrow{\sigma}$ is in ${\overrightarrow{\mathfrak B}}(\Phi^{eq}(A_r))$.
We call such $ \overrightarrow{\sigma}$ a {\textit flag basis}; there are $r!$ such sequences $\overrightarrow{\sigma}$.

\begin{lemma}
The set  $\overrightarrow{\mathcal  D}(A_r)$ consisting of flag bases is a diagonal subset of
${\overrightarrow{\mathfrak B}}(\Phi^{eq}(A_r))$.
\end{lemma}

We only need to prove that if $\overrightarrow{\sigma}$ and
$\overrightarrow{\tau}$ are two flag bases, then
$\Res^{\overrightarrow{\sigma}}f_\tau=0$ unless
$\overrightarrow{\sigma}=\overrightarrow{\tau}$. But this is evident.

\subsubsection{Systems of type $B_r$ or $C_r$}\label{B}
We consider $V=\R^r$ with standard basis $\{e^1,e^2,\ldots,e^r\}$.

Let $$B_r=[\pm e_i,\pm (e_i\pm e_j),\, 1\leq i \leq r,\,1\leq i<j\leq r]$$ be the root system of type $B$
and rank $r$.

Let
$$C_r=[\pm 2 e_i,\pm (e_i\pm e_j),\,  1\leq i \leq r,\,  1\leq i<j\leq r]$$
be the root system of type $C$ and rank $r$.

As roots of systems of type $B$ and $C$ are proportional, the system of hyperplanes  in $U={\mathfrak h}_\R^*$ defined by coroots of $B$ and $C$ are the same,
and we denote it by $\CH_r^{BC}$. More precisely, let $z=\sum_{i=1}^r z^i e_i$ in $U$, then the system of hyperplanes
$\CH_r^{BC}$   in $U$ is given by
$$\CH_r^{BC}=\cup_{1\leq i<j\leq r}\{z^i\pm z^j=0\}\cup \cup_{1\leq i\leq r} \{z^i=0\}.$$

We take the set
$$\Phi^{eq}(BC_r)=\cup_{1\leq i<j\leq r}\{e^i\pm e^j=0\}\cup \cup_{1\leq i\leq r} \{e^i=0\}$$
as equations of $\CH_r^{BC}$.

\bigskip

We define similarly a flag basis $\overrightarrow{\sigma}$ of $\Phi^{eq}(BC_r)$.
This is a basis of the form $\overrightarrow{\sigma}=[\alpha_1,\alpha_2,\ldots,\alpha_r]$ of $r$ elements of $\Phi^{eq}(BC_r)$  so that
$\alpha_i=e^i$,  or $e^i-e^j$  or $e^i+e^j$ with $j<i$.
That is, $\alpha_1=e^1$, $\alpha_2=e^2$ or $e^2-e^1$, or $e^2+e^1$,
 $\alpha_3=e^3$ or $e^3-e^2$, or $e^3+e^2$, or $e^3-e^1$, or $e^3+e^1$, etc.
Clearly, there are $(1)(3)(5)\cdots(2r-1)$ such sequences $\overrightarrow{\sigma}$.

\begin{lemma}
The set  $\overrightarrow{\mathcal  D}(BC_r)$ consisting of flag bases  is a diagonal subset of ${\overrightarrow{\mathfrak B}}(\Phi^{eq}(BC_r))$.
\end{lemma}

\begin{proof}
We first prove, by induction on $r$, that simple fractions $f_b$ associated to a flag basis $b$ generate $\CS_{\CH_r^{BC}}$.
We use the identities
$$\frac{1}{(x_r-x_i)}\frac{1}{(x_r+x_i)}=\frac{1}{(x_r+x_i)}\frac{1}{2x_r}+\frac{1}{(x_r-x_i)}\frac{1}{2x_r},$$
$$\frac{1}{x_r}\frac{1}{(x_r+x_i)}=-\frac{1}{(x_r+x_i)}\frac{1}{x_i}+\frac{1}{x_r}\frac{1}{x_i},$$
$$\frac{1}{x_r}\frac{1}{(x_r-x_i)}=\frac{1}{(x_r-x_i)}\frac{1}{x_i}-\frac{1}{x_r}\frac{1}{x_i},$$
to reduce to the case where a simple fraction
$f_b$ contains a linear form of type $e^r$, or $e^r+e^i$ or $e^r-e^i$ in the denominator, but not any two at the same time.
Then, by induction on $r$,
we see that a simple fraction $f_b$ associated to flag basis $b$ generates the space $\CS_{\CH_r^{BC}}$.
The dual property  on the elements of  $\overrightarrow{\mathcal  D}(BC_r)$  is evident.
\end{proof}

\begin{remark}
Although the  system $\Phi^{eq}(BC_r)$ is not unimodular for the lattice $\Lambda=\oplus \Z e^i$, we see that any $ \overrightarrow{\sigma}$ in the set $\overrightarrow{\mathcal  D}(BC_r)$ above is unimodular, so that the computation of
$\SZ^{\Lambda}(\tau)(f_\sigma)$ is easy.
\end{remark}

\subsubsection{The system of type $D_r$}\label{D}
We consider $V=\R^r$ with standard basis $\{e^1,e^2,\ldots,e^r\}$.
Let $$D_r=[\pm (e_i\pm e_j);\, 1\leq i<j\leq r]$$ be the root system of type $D$ and rank $r$.

Let $z=\sum_{i=1}^r z^i e_i$ in $U$.
We consider the system of hyperplanes
$$\CH_r^D=\cup_{1\leq i<j\leq r}\{ z^i\pm z^j=0\}.$$
The dimension of $\CS_{\CH_r^D}$ is known to be $(1)(3)(5)\cdots(2r-3)(r-1)$.
However, we did not find a nice diagonal basis for $\CS_{\CH_r^D}$.  Instead, we proceed as follows.

\bigskip

The set  $U_{\CH_r^D}$  of regular elements for $\CH_r^D$ contains $U_{\CH_r^{BC}}$.
Indeed, for any $z$ in $U_{\CH_r^D}$, we have $z^i\pm z^j\neq 0$, but $z^i$ may equal zero.

We define the set
$$U_{k,r}:=\{z^k=0,\, z^i\pm z^j \neq 0\; \text{for}\;1\leq i<j\leq r \; \text{and}\; z^i\neq 0 \;\text{for}\; 1\leq i\leq r, \,i\neq k\}.$$
Then, we have the following disjoint decomposition
$$U_{\CH_r^D}=U_{\CH_r^{BC}}\bigcup \cup_{k=1}^r U_{k,r}.$$
The set $U_{k,r}$ is clearly isomorphic to the open set $U_{\CH^{BC}_{r-1}}$ in rank $r-1$ via the map
$i_k$ which inserts a zero coordinate in position $k$, and hence,

\begin{equation}\label{eq:disjoint}
U_{\CH_r^D}=U_{\CH_r^{BC}}\bigcup \cup_{k=1}^r i_k(U_{\CH^{BC}_{r-1}}).
\end{equation}
The above decomposition allows us to reduce calculations in system of type $D$ to that of systems of type $B$ or $C$.

\subsection{Calculations of multiple Bernoulli series for type  $A$}
We use the same notation as in Section \ref{A}.

Let $Q_A\subset U$ be the root lattice generated by $A_r$,
and $P_A\subset U$ be the weight lattice. Then $P_A$  is  generated by $Q_A$ and $e_1-\frac{1}{r+1}(e_1+e_2+\cdots +e_{r+1})$ and
 $P_A/Q_A$ is of cardinality $r+1$.

Let $\Gamma$ be a lattice such that  $Q_A\subset \Gamma\subset P_A$.
We denote by $\Gamma_{reg}=\Gamma\cap U_{\CH_r^A}$ the set of regular elements in $\Gamma$.

Let $\Lambda\subset V$ be the dual lattice to $\Gamma$, and
let ${\bf s}=[s_{\alpha}]$ be a list of exponents.  Define
$$g^A_{\bf s}(z)=\frac{1}{\prod_{\alpha>0}\ll H_\alpha,z\rr^{s_\alpha}},$$
where the set $\{H_\alpha, \alpha>0\}$ is the set of positive coroots $\Phi^{eq}(A_r)$.

If $v\in V$,
$$\mathcal B(\CH_r^A,\Lambda, g^A_{\bf s})(v)=\sum_{\gamma\in \Gamma_{reg}} \frac{e^{2i\pi \ll v,\gamma\rr}}{\prod_{\alpha>0} (2i\pi \ll H_\alpha,\gamma \rr)^{s_\alpha}}.$$

If we use the diagonal basis $\overrightarrow{\mathcal  D}_W$ for $\Phi^{eq}(A_r)$ as defined in Proposition \ref{pro:OW}, then the diagonal basis consists of elements
$[e^{w(1)}-e^{w(2)},\ldots, e^{w(r)}-e^{r+1}]$ where $w$ is a permutation leaving $e^{r+1}$ stable.
Thus if we express $v=\sum_{i=1}^r v_i (e^{i}-e^{r+1})$, the algebra $Step(\overrightarrow{\mathcal  D}_W)$ consists of functions $\{\sum_I v_i\}$ where $I$
runs over subsets of $\{1,2,\ldots, r\}$.

We now discuss two simple cases, where $\Gamma$ is either the weight lattice $P_A$, or the root lattice $Q_A$.

\subsubsection{Bernoulli series for the weight lattice}

If $P_A$ is the weight lattice, then the dual of $P_A$ is the coroot lattice $\check{Q}_A$
generated by simple coroots $H_\alpha$,
% (the set $\Sigma$?),
and the system $\Phi^{eq}(A_r)$  of equations (the positive coroots) is unimodular with respect to $\check{Q}_A$.

Thus $\mathcal B(\CH_r^A,\check{Q}_A,g^A_{\bf s})(v)$ is a piecewise polynomial function of degree
$\sum_\alpha s_\alpha$ and lies in the algebra $Step(\overrightarrow{\mathcal  D}_W)$.
Our program  then gives
$\mathcal B(\CH_r^A,\check{Q}_A,g^A_{\bf s})(v)$ as a polynomial expression of the functions   $\{\sum_I v_i\}$.
It also computes numerically the value of this function at any point $v$.

\begin{example}\label{a2exp1}
Consider the root system of type $A$ and of rank $r=2$. We assume all multiplicities $s_\alpha=1$, and compute
$\mathcal B(\CH_2^A,\check{Q}_A, g^A_{\bf s})(v)$ for $v=v_1 e^1+ v_2 e^2+ v_3 e^3$ with $v_1+v_2+v_3=0$.
The simple coroots are $H_{\alpha_1}=e^1-e^2$ and $H_{\alpha_2}=e^2-e^3$, and the remaining positive coroot is their sum
$H_{\alpha_1}+H_{\alpha_2}$.
The dual lattice has basis dual to $H_{\alpha_1}$, $H_{\alpha_2}$.
Thus, for ${\bf s}=[1,1,1],$
\begin{equation}\label{sumsone}
\mathcal B(\CH_2^A,\check{Q}_A, g^A_{\bf s})(v)=\sum'_{m,n \in \Z}\frac{e^{2i\pi m v_1-2i\pi n v_3}}{(2i\pi m)(2i\pi n)(2i\pi (m+n))}.
\end{equation}
The symbol $\sum'$ above means that we sum over the integers $m,n$ with $mn(m+n)\neq 0$.

Denoting the fractional part of $t$ with $\{t\}\in[0,1[$ as before, we obtain
that $P(v_1,v_2,v_3)=\mathcal B(\CH_2^A,\check{Q}_A, g^A_{\bf s})(v_1 e^1+v_2 e^2+v_3 e^3)$ is equal to
{\small \begin{equation}\label{weight}
\frac{1}{6}
(\{v_2\}-\{v_1\})(\{v_1\}^2-3\{v_1+v_2\}\{v_1\}+\{v_2\}\{v_1\}+3\{v_1+v_2\}-1-3\{v_1+v_2\}\{v_2\}+\{v_2\}^2).
\end{equation}}
We remark that the series (\ref{sumsone}) is not absolutely convergent, but the sum has a meaning and is a piecewise polynomial function.
\end{example}
\medskip

Let us give some numerical  examples.
Consider again $A_2$.  Suppose ${\bf s}=[10,10,10]$ and $v_1=v_2=0$.  Then
$$\mathcal B(\CH_2^A,\check{Q}_A,g^A_{\bf s})(0)=\sum'_{m,n}\frac{1}{(2i\pi m)^{10}(2i\pi n)^{10}(2i\pi (m+n))^{10}}$$
is equal to
$$-{\frac {27739097}{4174671932121099276691439616000000}}.$$

\medskip

Consider now the system of type $A$ and rank $r=4$.  Suppose $v=[0,0,0,0,0]$.
%with  exponents  given by ${\bf s}=[6,6,6,6,4,2,2,2,2,2]$.  Here we list the exponents
We list the exponents
with respect to the following order on the roots
$$[e_1-e_2, e_1-e_3, e_1-e_4, e_1-e_5, e_2-e_3, e_2-e_4, e_2-e_5, e_3-e_4, e_3-e_5, e_4-e_5].$$ Then
for ${\bf s}=[6,6,6,6,4,2,2,2,2,2]$ we have
$$\mathcal B(\CH_4^A,\check{Q}_A,g^A_{\bf s})(0)=\frac{1}{(2i\pi)^{38}}\times $${\tiny{$$\sum'\frac{1}{m_1^6m_2^4m_3^2m_4^2
(m_1+m_2)^6(m_1+m_2+m_3)^6(m_1+m_2+m_3+m_4)^6(m_2+m_3)^2(m_2+m_3+m_4)^2(m_3+m_4)^2}$$}}
{\small $$=\frac{66581757}{2081416538897698301902069565296214016000000000},$$}
while
for  ${\bf s}=[4,4,4,4,4,4,4,4,4,4]$ we obtain
{\small $$\mathcal B(\CH_4^A,\check{Q}_A,g^A_{\bf s})(0)=\frac{3998447009863}{19318834119102098604968210835862034086625280000000000}.$$}

\bigskip

\subsubsection{Bernoulli series for the root lattice}

Let $\xi=\sum_{j=1}^r (e^j-e^{r+1})$. Then a system of representatives
for $\check{P}_A/\check{Q}_A$ consists of the elements $\lambda_j=\frac{j}{r+1} \xi$, with $j$ varying between $0$ and $r$.
Thus, using Formula (\ref{eq:lattices}),
$${\mathcal B}(\CH_r^A,\check{P}_A, g^A_{\bf s})(v)=\frac{1}{r+1}\sum_{j=0}^{r} {\mathcal B}(\CH_r^A ,\check{Q}_A, g^A_{\bf s})(v+\lambda_j).$$
We obtain  an expression for
${\mathcal B}(\CH_r^A,\check{P}_A, g^A_{\bf s})$ in terms of the functions $\{(\sum_I v_i)+c/(r+1)\}$ where $c$ are integers between $0$ and $r$.

Recall Example \ref{a2exp1}. For the same data, we now compute
${\mathcal B}(\CH_2^A,\check{P}_A, g^A_{\bf s})$ for $v=v_1 e^1+ v_2 e^2+ v_3 e^3$ with $v_1+v_2+v_3=0$.
Hence, we express  $v=v_1(e^1-e^2)-v_3(e^2-e^3)$, and compute for  ${\bf s}=[1,1,1]:$
$${\mathcal B}(\CH_2^A,\check{P}_A, g^A_{\bf s})(v)=\sum'_{m,n}\frac{e^{2i\pi m (v_1-v_2)+2i\pi n (v_2-v_3)}}
{(2i\pi (2m-n))(2i\pi (2n-m))(2i\pi (m+n))}.$$
This is equal to
$$\frac{1}{3}\left(P(v_1,v_2,v_3)+P(v_1+\frac{1}{3},v_2+\frac{1}{3},v_3-\frac{2}{3})+
P(v_1+\frac{2}{3},v_2+\frac{2}{3},v_3-\frac{4}{3})\right)$$
where $P$ is the piecewise polynomial function given in Equation (\ref{weight}).

\subsection{Calculation of multiple Bernoulli series for types $C$ and $B$}
We use the same notation as in Section \ref{B}.

\medskip

We now consider the system $\CH_r^{BC}$ of hyperplanes in $U$
$$ \CH_r^{BC}=\cup_{1\leq i\leq r}\{z^i=0\}\cup \cup_{1\leq i<j\leq r}\{z^i\pm z^j=0\}.$$

Let $\Lambda$  be a lattice commensurable with $\oplus \Z e^i$,  with dual lattice $\Gamma$.
Denote simply by  $\Gamma_{reg}=\Gamma_\reg(\CH_r^{BC})$.
If $g\in \CR_{\CH_r^{BC}},$
\begin{equation}\label{typeC}
\mathcal B(\CH_r^{BC},\Lambda,g)(v)=\sum_{\gamma\in \Gamma_{\reg}}
 g(2i\pi \gamma)e^{2i\pi \ll v,\gamma\rr}.
\end{equation}

\subsubsection{Root system $C_r$}

Let $P_C$ be the weight lattice of the root system $C_r$.
We thus have the coroot lattice $\check{Q}_C=\oplus_{i=1}^r \Z e^i$.

Let ${\bf s}=[s_{\alpha}]$ be a list of exponents and
let $$g_{\bf s}^C(z)=\frac{1}{\prod_{\alpha>0}\ll H_\alpha,z\rr^{s_\alpha}}.$$
Here $\{H_\alpha, \; \alpha>0\}$ are positive coroots of the system $C_r$, which are explicitly
$$\{ e^i, 1\leq i\leq r, (e^i\pm e^j), 1\leq i<j\leq r\}.$$
Clearly, the function $g_{\bf s}^C$ belongs to $\CR_{\CH_r^{BC}}$.

If $v\in V$,
$$\mathcal B(\CH_r^{BC},\check{Q}_C, g_{\bf s}^C)(v)=\sum_{\gamma\in (P_C)_{reg}} \frac{e^{2i\pi \ll v,\gamma\rr}}{\prod_{\alpha>0} (2i\pi \ll H_\alpha,\gamma \rr)^{s_\alpha}}.$$

The function $\mathcal B(\CH_r^{BC},\check{Q}_C, g_{\bf s}^C)$ is a piecewise polynomial function on $V$ of degree
$\sum_\alpha s_\alpha$.
We  use the diagonal basis  constructed in Section \ref{B} to compute it.
Let us now compute the example given in the introduction, which corresponds to $C_2$,
and the exponent ${\bf s}=[2,1,1,1]$ ordered in accordance with the order $[2e_1,2 e_2,e_1+e_2,e_1-e_2]$ of
positive roots (so that coroots $H_\alpha$ are  in order  $[e^1, e^2,e^1+e^2,e^1-e^2]$).
We express $v=v_1 e^1+v_2 e^2$ and compute $$\mathcal B(\CH_2^{BC},\check{Q}_C, g^C_{\bf s})(v)=
\sum'_{m,n}\frac{e^{2i\pi m v_1+2i\pi n v_2}}{(2i\pi m)^{2}(2i\pi n)(2i\pi (m+n))(2i\pi (m-n))}.$$

This piecewise polynomial function  is given by
{\small \begin{equation}\label{eq:BB2} \end{equation}
$$\mathcal B(\CH_2^{BC},\check{Q}_C, g^C_{\bf s})(v)= Q(v_1,v_2)=$$$$
-{\frac {1}{160}}\,{ \left\{ -v_{{2}}+v_{{1}} \right\} }^{5}-1/48\,{
 \left\{ v_{{1}} \right\} }^{2}+1/24\,{ \left\{ v_{{1}} \right\} }^{3}
+1/24\,{ \left\{ v_{{2}}+v_{{1}} \right\} }^{3} \left\{ v_{{2}}
 \right\} -$$$$1/48\,{ \left\{ v_{{2}}+v_{{1}} \right\} }^{4} \left\{ v_{{
2}} \right\} -1/48\,{ \left\{ v_{{2}}+v_{{1}} \right\} }^{2} \left\{ v
_{{2}} \right\} -{\frac {1}{960}}\, \left\{ v_{{2}}+v_{{1}} \right\} +
{\frac {1}{96}}\,{ \left\{ v_{{2}}+v_{{1}} \right\} }^{2}-$$$${\frac {1}{
96}}\,{ \left\{ v_{{2}}+v_{{1}} \right\} }^{3}-{\frac {1}{192}}\,{
 \left\{ v_{{2}}+v_{{1}} \right\} }^{4}+{\frac {1}{960}}\, \left\{ -v_
{{2}}+v_{{1}} \right\} +{\frac {1}{96}}\,{ \left\{ -v_{{2}}+v_{{1}}
 \right\} }^{2}-$$$$1/32\,{ \left\{ -v_{{2}}+v_{{1}} \right\} }^{3}+{
\frac {5}{192}}\,{ \left\{ -v_{{2}}+v_{{1}} \right\} }^{4}-1/48\,{
 \left\{ v_{{1}} \right\} }^{4}+1/24\,{ \left\{ v_{{1}} \right\} }^{2}
 \left\{ v_{{2}} \right\} -$$$$1/12\,{ \left\{ v_{{1}} \right\} }^{3}
 \left\{ v_{{2}} \right\} +1/24\,{ \left\{ v_{{1}} \right\} }^{4}
 \left\{ v_{{2}} \right\} +{\frac {1}{160}}\,{ \left\{ v_{{2}}+v_{{1}}
 \right\} }^{5}+1/24\,{ \left\{ -v_{{2}}+v_{{1}} \right\} }^{3}
 \left\{ v_{{2}} \right\} -$$$$1/48\,{ \left\{ -v_{{2}}+v_{{1}} \right\} }
^{4} \left\{ v_{{2}} \right\} -1/48\,{ \left\{ -v_{{2}}+v_{{1}}
 \right\} }^{2} \left\{ v_{{2}} \right\}.
$$}

%Let us give an example for $C_2$.
%We write $v=v_1 e^1+v_2 e^2$.
%We thus compute:

%$$\mathcal B(\CH_r^{BC},\check{P_C}, g^C_{\bf s})(v)=
%\sum'_{m,n}\frac{e^{2i\pi m v_1+2i\pi n v_2}}{(2i\pi m)^{s_1}(2i\pi n)^{s_2}
%(2i\pi (m+n))^{s_3} (2i\pi (m-n))^{s_4}}.$$

%For all the exponents $s_\alpha$ equal to $1$, this piecewise polynomial function  is
%\begin{equation}\end{equation}\label{exB2}

%$$Q(v_1,v_2)=\frac{1}{8}(-1/3\, \left\{ v_{{1}} \right\} +2/3\, \left\{ v_{{1}} \right\}
% \left\{ v_{{2}} \right\} +{ \left\{ v_{{1}} \right\} }^{2}-2\,{
% \left\{ v_{{1}} \right\} }^{2} \left\{ v_{{2}} \right\} -2/3\,{
% \left\{ v_{{1}} \right\} }^{3}+$$$$4/3\,{ \left\{ v_{{1}} \right\} }^{3}
%\left\{ v_{{2}} \right\} +1/6\, \left\{ v_{{2}}+v_{{1}} \right\} -1/3
%\, \left\{ v_{{2}}+v_{{1}} \right\}  \left\{ v_{{2}} \right\} -1/4\,{
% \left\{ v_{{2}}+v_{{1}} \right\} }^{2}+$$$${ \left\{ v_{{2}}+v_{{1}}
 %\right\} }^{2} \left\{ v_{{2}} \right\} -1/6\,{ \left\{ v_{{2}}+v_{{1
%}} \right\} }^{3}-2/3\,{ \left\{ v_{{2}}+v_{{1}} \right\} }^{3}
% \left\{ v_{{2}} \right\} +$$$$
%1/4\,{ \left\{ v_{{2}}+v_{{1}} \right\} }^{
%4}+1/6\, \left\{ -v_{{2}}+v_{{1}} \right\} -1/3\, \left\{ -v_{{2}}+v_{
%{1}} \right\}  \left\{ v_{{2}} \right\}$$$$ -3/4\,{ \left\{ -v_{{2}}+v_{{1
%}} \right\} }^{2}+{ \left\{ -v_{{2}}+v_{{1}} \right\} }^{2} \left\{ v_
%{{2}} \right\} +$$$$5/6\,{ \left\{ -v_{{2}}+v_{{1}} \right\} }^{3}-2/3\,{
% \left\{ -v_{{2}}+v_{{1}} \right\} }^{3} \left\{ v_{{2}} \right\} -1/4
%\,{ \left\{ -v_{{2}}+v_{{1}} \right\} }^{4}).$$

\medskip

Let's see what happens  on a tope.   Figure \ref{b2} depicts topes associated to the pair
$\Phi^{eq}(BC_2)=\{e^1, e^2, e^1+e^2, e^1-e^2\}$ and
$\Lambda=\Z e^1 \oplus \Z e^2$.

Consider for example the tope $$\tau_2=\{v_1>0,\, v_2>0,\, v_1>v_2,\, v_1+v_2<1\}.$$
Then on $\tau_2$, the piecewise polynomial function  $\mathcal B(\CH_2^{BC},\check{Q}_C, g^C_{\bf s})(v)$ coincides with the polynomial
%\begin{equation}
%\frac{1}{8}(v_{{1}}v_{{2}}-3\,{v_{{1}}}^{2}v_{{2}}+2\,{v_{{1}}}^{3}v_{{2}}-{v_{{2}
%}}^{2}+2\,v_{{1}}{v_{{2}}}^{2}-2\,v_{{1}}{v_{{2}}}^{3}+{v_{{2}}}^{3}).
%\end{equation}
\begin{equation} \mathcal B(\CH_2^{BC},\check{Q}_C, g^C_{\bf s})(v)=Q_{\tau_2}(v_1,v_2)=\end{equation}
{\tiny
$$\frac{1}{8}(-{\frac {1}{60}}\,v_{{2}}+1/2\,{v_{{1}}}^{2}v_{{2}}-{v_{{1}}}^{3}v_{{2}}+1/6\,{v_{{2}}}^{2}-v_{{1}}{v_{{2}}}^{2}+v_{{1}}{v_{{2}}}^{3}
+{v_{{1}}}^{2}{v_{{2}}}^{2}-1/6\,{v_{{2}}}^{3}+1/6\,{v_{{2}}}^{4}+1/2\,{v_{{1}}}^{4}v_{{2}}-{v_{{1}}}^{2}{v_{{2}}}^{3}-{\frac {7}{30}}\,
{v_{{2}}}^{5})$$}
If we compute $Q_{\tau_2}(v_1,v_2)$ for $v_1=\frac{1}{15},\ v_2=\frac{1}{30}$ we obtain
$-\frac{276037}{5832000000}$.
%which is the same value that we obtain
%for $Q_{tau}(\frac{1}{15},\frac{1}{30})$ as it should be, since $(\frac{1}{15},\frac{1}{30})\in \tau$

\medskip

We give some more numerical examples with different exponents.

\medskip

For example, we may compute with exponents ${\bf s}=[s_1,s_2,s_3,s_4]$ associated to the order of roots $[2e_1,2 e_2,e_1+e_2,e_1-e_2]$
  and $v=[v_1,v_2]$

{\small \[\mathcal B(\CH_2^{BC},\check{Q}_C, g^C_{\bf s})(v_1e^1+v_2e^2)=\sum'_{m,n}\frac{e^{2i\pi mv_1+2i\pi n v_2}}
{(2i\pi m)^{s_1}(2i\pi n)^{s_2}
(2i\pi (m+n))^{s_3} (2i\pi (m-n))^{s_4}}\]}
\[=\left\{ \begin{array}{l@{\quad if \quad } l@{\quad and\quad} l}
\frac{810650239}{132316540312500}&  ${\bf s}=[2,2,1,1]$&$v=[1/5, 1/19]$\\
\frac{47036110438854761301636459941}{1529174429579197250943325345977126782238720}&${\bf s}=[2,3,4,5]$
& $v=[1/7, 1/17]$
\end{array}\right.\]

\bigskip

\subsubsection{Root system $B_r$}

We consider  the root system of type $B$ and rank $r$. Let $\Gamma=P_B$ be the lattice of weights of $B$, and
as before the dual lattice generated by the coroots is denoted by $\check{Q}_B$.

Let  ${\bf s}=[s_{\alpha}]$ be a list of exponents  and
$$g^B_{\bf s}(z)=\frac{1}{\prod_{\alpha>0}\ll H_\alpha,z\rr^{s_\alpha}}.$$
Here, again,  $\{H_\alpha, \; \alpha>0\}$ are positive coroots of the system $B_r$.

If $v\in V$, ${\bf s}=[s_{\alpha}]$ then
$$\mathcal B(\CH_r^{BC},\check{Q}_B, g_{\bf s}^B)(v)=\sum_{\gamma\in (P_B)_{reg}} \frac{e^{2i\pi \ll v,\gamma\rr}}{\prod_{\alpha>0} (2i\pi \ll H_\alpha,\gamma \rr)^{s_\alpha}}.$$

Clearly, as long coroots of $B$  are twice the short coroots of $C$, and short coroots of $B$ are long coroots of $C$, we have
$$g_{\bf s}^B=c_1 g_{\bf s}^C, \ \text {where} \ c_1=\frac{1}{2^{s_{2e_1}+\cdots+s_{2e_r}}}.$$
We then use the comparison formula with two lattices.
Indeed, we have $2 \check{Q}_C\subset  \check{Q}_B.$
The lattice
$2 \check{Q}_C$ is of index
$2^{r-1}$  in  $\check{Q}_B$ and a set  of representatives  is given, for example, by
$$F :=\{0,e^{i_1}+e^{i_2}+\cdots +e^{i_k}, 1\leq i_1< i_2<\cdots <i_k,\ k=2j,  \ 1\leq j\leq [r/2]\}.$$

We then use Formulae (\ref{eq:lattices}) and (\ref{eq:dillattices}).
Since $g^C_{\bf s}$ is homogeneous of degree $- \sum_\alpha s_{\alpha}$, we obtain
$$\mathcal B(\CH_r^{BC},\check{Q}_B,g^B_{\bf s})(v)=
\frac{1}{2^{r-1}}c_2\left(\sum_{\lambda\in F}\mathcal B(\CH_r^{BC},\check{Q}_C,g^C_{\bf s})({ \frac{v+\lambda}{2}})\right)$$
where $c_2=2^{\sum_{1\leq i<j\leq r} s_{e_i-e_j}+s_{e_i+e_j}}$.
In particular, if  ${\bf s}=[m,\ldots,m]$, then $c_2=2^{r(r-1)m}.$

For example, for $B_2$, we compute for $v=v_1 e^1+ v_2 e^2$  and  multiplicities ${\bf s}=[2,1,1,1]$ with respect to the order $[e_1-e_2, e_2, e_1+e_2, e_1]$ of roots,
$$(2i\pi)^5\mathcal B(\CH_2^{BC},\check{Q}_B,g^B_{\bf s})({v})=\sum'_{m_1,m_2}\frac{e^{2i\pi( \left( m_1+1/2\,m_2 \right) v_1+1/2\,m_2v_2
)}}{m_1^2m_2\left(2 m_1+\,m_2 \right)  \left( m_2+m_1
\right) } $$
where the symbol $\sum'$ means that we sum over the $m_1,m_2$ with $$\ \left( 2m_{{1}}+\,m_{{2}} \right)  \left( m_{{2}}+m_{{1}}
\right) m_{{2}}m_{{1}}\neq 0.$$

We obtain
$$\mathcal B(\CH_2^{BC},\check{Q}_B,g^B_{\bf s})({{v}})=2\left(Q(\frac{v_1}{2},\frac{v_2}{2})+Q(\frac{v_1+1}{2},\frac{v_2+1}{2})\right)$$
where $Q=\mathcal B(\CH_2^{BC},\check{Q}_C,g^C_{\bf s})$ is given in Equation (\ref{eq:BB2}).

In particular, for $u=[1/15,1/30]$ we obtain
$$\mathcal B(\CH_2^{BC},\check{Q}_B,g^B_{\bf s})({{u}})=\frac{-276037}{5832000000}.$$

\medskip

For $B_3$,
 we compute for $v=v_1 e^1+ v_2 e^2+ v_3 e^3$ and ${\bf s}=[1,1,1,1,1,1,1,1,1]:$
{\tiny{$$(2i\pi)^{9}\mathcal B(\CH_3^{BC},\check {Q_B},g^B_{\bf s})({{v}})=$$$$\tiny{\sum'_{m_1,m_2,m_3}\frac{e^{2i\pi( \left( m_1+m_2+1/2\,m_3 \right) v_1+ \left( m_2+1
/2\,m_3 \right) v_{{2}}+1/2\,m_3v_3
)}}{ \, \left( 2m_{{1}}+2 m_{{2}}+\,m_{{3}} \right)  \left( 2m_{{2}}+\,m
_{{3}} \right) m_{{3}}m_{{1}} \left( m_{{1}}+2\,m_{{2}}+m_{{3}}
 \right)  \left( m_{{1}}+m_{{2}} \right)  \left( m_{{1}}+m_{{2}}+m_{{3
}} \right) m_{{2}} \left( m_{{2}}+m_{{3}} \right)
}}.$$}}

We obtain
$$\mathcal B(\CH_3^{BC},\check {Q}_B,g^B_{\bf s})({v})=$$
{\small $$2^4\left(S(\frac{v_1}{2},\frac{v_2}{2},\frac{v_3}{2})+S(\frac{v_1+1}{2},\frac{v_2+1}{2},
\frac{v_3}{2})+S(\frac{v_1+1}{2},\frac{v_2}{2},\frac{v_3+1}{2})+S(\frac{v_1}{2},\frac{v_2+1}{2},\frac{v_3+1}{2})\right)$$}
where $S=\mathcal B(\CH_3^{BC},\check {Q}_C,g^C_{\bf s})$ is a piecewise polynomial that is too long to be included here.

\medskip

\subsection{Calculation of multiple Bernoulli series for type $D$}

We follow the same notation as in Section \ref{D}.

Let $\mathcal{H}_r^D=\cup_{i,j}\{z^i\pm z^j=0 ,1\leq i<j\leq r\}$, and let ${\bf s}=[s_\alpha]$ be a list of exponents.
The ordering of elements in the list ${\bf s}$ is taken to match the following ordering
$[e_1-e_2,e_1-e_3,\ldots,e_1-e_r,e_2-e_3,\ldots, e_1+e_2,\ldots,e_{r-1}+e_r]$ of positive roots of the system $D_r$.

Consider
$$g^D_{\bf s}(z)=\frac{1}{\prod_{\alpha>0}\ll H_\alpha,z\rr^{s_\alpha}},$$
where  $\{H_\alpha, \; \alpha>0\}$ are positive coroots of $D_r$.

We  embed the list of roots of $D_r$ in to the list of roots of $B_r$ by writing the short roots $e_i$ of $B_r$ at the
very end of the list.  We denote by ${\bf{S}}=[s_\alpha,0,\ldots, 0]$  the list obtained from  ${\bf s}$ by adjoining
$r$ zeros to its end.  Then, by construction, we have
$$g^D_{\bf s}(z)=g^B_{\bf S}(z).$$

We now associate to the list $\bf{s}$  a list of exponents
${\bf{s}}_k$ for the system $B$ of rank $r-1$. In $\bf{s}$ we
eliminate the position  corresponding to the roots $e_i\pm e_k$ for
$i<k$, and $e_k\pm e_i$ for $i>k$.  Then we assign the value
$s_{e_i+e_k}+s_{e_i-e_k}$ to the exponent corresponding to root $e_i$ of $B_{r-1}$
for $i<k$ , similarly we assign the value $s_{e_k+e_i}+s_{e_k-e_i}$ to the exponent corresponding to the root $e_i$ of
$B_{r-1}$ for $i>k$.

We also let $i_k(v)$ to be the vector with $r-1$ coordinates obtained from $v=\sum_{i=1}^r v_i e^i$  by putting $v_k=0$.

Let $\Gamma=P_D$ be the weight lattice of $D$ and $\check{Q}_D$ the dual lattice generated by the coroots.
Since $P_D$ is the weight lattice of the simply connected group $Spin(2r)$,
$\gamma=\sum_{i=1}^r \gamma^i e_i$ is in $P_D$ if
$\gamma^i\pm \gamma^j\in \Z$ and $P_D=P_B$.
Consider the intersection of $P_D$ with the hyperplane $z^k=0$.
Then, we see that this intersection is isomorphic to the weight lattice of a system $C_{r-1,k}$  of type $C$, rank $r-1$, embedded in $C$ of rank $r$
with simple roots $\{e_1-e_2,e_2-e_3,\ldots,e_{k-1}-e_{k+1},e_{k+1}-e_{k+2},\ldots, e_r\}.$

Using the decomposition (\ref{eq:disjoint}), we decompose the set of regular elements of the lattice $P_D$ as a disjoint
union of the
set of regular elements of the lattice $P_B$ and the set
of regular elements of the lattice $P_C$ of rank $r-1$.

In particular, if $\gamma \in (P_{D})_{reg}$ and $\gamma_k=0$, then
\[\frac{1}{{\prod \atop \alpha >0} ( 2i\pi \ll H_{\alpha},\gamma \rr)^{s_{\alpha}}}=\]
\[\frac{1}{{\prod \atop \gamma_i\pm \gamma_j \neq 0,\gamma_k=0}
(2i\pi (\gamma_i- \gamma_j ))^{s_{e_i-e_j}}(2i\pi (\gamma_i+ \gamma_j ))^{s_{e_i+e_j}}}= c_kg^C_{{\bf s}_k}(i_k(\gamma))\]
with $c_k=(-1)^{\sum_{j=k+1}^r{s_{e_k-e_j}}}.$
%and  $\hat{s}_{e_i-e_k}={s_{e_i-e_k}}$ if $i<k,$ $\hat{s}_{e_i-e_k}={s_{e_k-e_i}}$ otherwise.
In particular, if ${\bf s}=[m,\ldots,m]$, then $c_k=(-1)^{(r-k)m}.$

Thus, we can compute multiple Bernoulli series for a system of type $D_r$ by using computations for types $B_r$  and
$C_{r-1}$ with appropriate exponents.  More explicitly,
\[\mathcal{B}(\mathcal{H}_r^D, \check{Q}_D,g^D_{\bf s})({v})=\sum_{\gamma \in (P_{D})_{reg}}  \frac{e^{ 2i\pi \ll v ,\gamma \rr}}{
\prod_{\alpha >0} (2i\pi \ll H_{\alpha},\gamma \rr)^{s_{\alpha}}}= \]
%\[\sum_{\gamma=\sum \gamma_ie_i, \gamma_i\pm \gamma_j \neq 0}  \frac{e^{\ll 2i\pi v ,\gamma \rr}}{
%\prod_{1\leq i<j\leq n} ((2i\pi))^K (\gamma_i\pm \gamma_j )^{K} }=\]
%[\sum_{\gamma \in \check{P_B},}  \frac{e^{\ll 2i\pi v ,\gamma \rr}}{
%\prod_{1\leq i<j\leq n} ((2i\pi))^K (\gamma_i\pm \gamma_j )^{K} }=\]
\[\mathcal{B}(\mathcal{H}_r^{BC}, \check{Q}_B,g^B_{\bf S})({v})+
%\[=\sum_{ \gamma\in L_{\neq 0}}  \frac{e^{\ll 2i\pi v ,\gamma \rr}}{
%\prod_{1\leq i<j\leq n}  ((2i\pi)^{K}\gamma_i\pm \gamma_j )^{K} }+\]
%\[\sum_{k=1}^n(-1)^{n-k}\sum_{\gamma\in L_{ k}}  \frac{e^{\ll 2i\pi v ,\gamma \rr}}
%{\prod_{1\leq i<j\leq n \ i,j\neq k}  ((2i\pi)^{K}\gamma_i\pm \gamma_j )^{K} \prod_{1\leq i\leq n \  \ i \neq k}  (\gamma_i )^{2K} }\]
\sum_{k=1}^r c_k\mathcal{B}(\mathcal{H}_{r-1}^{BC}, {\check{Q}}_{C_{r-1,k}},g^C_{{\bf s}_k})({i_k(v)}).\]

%
% \[ \mathcal{B}(\mathcal{H}^D, \check{P_D},g^D_{\bf s})({v})= \mathcal{B}(\mathcal{H}^{BC}, \check{P_B},g^B_{\bf S})({v})+\sum_{k=1}^r (-1)^{r-k} \mathcal{B}(\mathcal{H}^{BC}, \check{P_C},g^C_{{\bf s}_{k}})(i_k(v))\]

%In the right hand side,  for the first term,  $\mathcal{H}^{BC}$ is the system of hyperplane  $\mathcal{H}^{BC}$  in rank $r$,
%while in the sum indexed by $k$, $\mathcal{H}^{BC}$
%is the system of hyperplane  $\mathcal{H}^{BC}$  in rank $r-1$.

%%We can now use the formula of the previous section to compute $\mathcal{B}(\mathcal{H}^B, \check{P_B},g^B_{\bf S})({v})$ and we obtain
%%$$\mathcal{B}(\mathcal{H}^D, \check{P_D},g^D_{\bf s})({v})=$$$$\frac{1}{2^{r-1}} C_2\left(\mathcal B(\CH_r^{BC},\check {P_C},g^C_{\bf {s}})({\bf\frac{ v}{2}})+\sum_{k=2j,j=1..[r/2]}\mathcal B(\CH_r^{BC},\check {P_C},g^C_{\bf s})({\bf \frac{v+\lambda(k)}{2}})\right)+$$$$\sum_{k=1}^n (-1)^{n-k} \mathcal{B}(\mathcal{H}^{BC}, \check{P_C},g^C_{{\bf s_{k}}})({\bf{i_k(v)}})$$
%%
%%where $C_2=2^{\sum_{1\leq i<j\leq r} s_{e_i-e_j}+s_{e_i+e_j}}$
%%In particular if $s=[m,\cdots,m]$ then $C_2=2^{r(r-1)m}.$
%
%%$\bullet$
%%Another interesting lattice to consider is the weight lattice of the group $SO(2r)$.
%%It is the lattice $\Gamma=\oplus_i \Z e^i$, with dual lattice
%%$\Lambda=\oplus_i \Z e_i$.
%%
%%
%%THEN...
%
%
%
\section{Witten formula for volumes of moduli spaces of flat connections on surfaces}

Let $G$ be a simple, simply connected, compact Lie group of rank $r$
with maximal torus $T$.  For $g_1,g_2\in G$, we denote by
$[g_1,g_2]=g_1g_2 g_1^{-1} g_2^{-1}$ the commutator of $g_1,g_2$.
Let $\Sigma$ be a compact connected oriented surface of genus $g$
and let $p:=\cup_j\{p_j\}$ be a set of  $s$ points on $\Sigma$. Let
$\mathcal C:=(\mathcal C_j)$ be a set of $s$ conjugacy classes in
$G$. We consider the representation variety
$$\mathcal M(G,g,s,\mathcal C):=\{(a,c)\in G^{2g}\times \mathcal C;\;\prod_{i=1}^g [a_{2i-1},a_{2i}]=\prod_{j=1}^s c_j\}/G.$$
If the adjoint orbits $\mathcal C_j$ are generic, this is an
orbifold of dimension $(2g-2)\dim G+s\dim G/T$. It parameterizes the
set of flat $G$-valued connections on $\Sigma-p$, with holonomy
around $p_j$ belonging to the conjugacy class $\mathcal C_j$ modulo
gauge equivalence. As shown by Atiyah-Bott \cite{ab}, once a
$G$-invariant inner product on $\mathfrak g$ is chosen, the manifold
$\mathcal M(G,g,s,\mathcal C)$ carries a natural  symplectic form,
and Witten gave a formula for the volume of $\mathcal
M(G,g,s,\mathcal C)$ that we recall.

\medskip

We use the notation of Section \ref{section:clroot}.  We need some more definitions.

\medskip

%For $\alpha \in \frh^*$, define
%$(\frg_{\C})_{\alpha}=\{X \in \frg_{\C}; [H,X]=\ll \alpha,H\rr X\;\text{for all}\;H \in\frh\}$.
%If $\alpha \neq 0$ and
%$(\frg_{\C})_{\alpha}\neq 0$, then $\alpha$ is called a {\it root} of
%$\frh$ in $\frg_{\C}$. Let $R=R(\frg_{\C},\frh)\subset \frh^*$ be
%the set of roots; denote the root lattice by $Q$. Roots $\alpha \in
%R$ take imaginary values on $\frt$.

%If $\alpha \in R$, there exists a unique element $H_{\alpha}$ in
%$[(\frg_{\C})_{\alpha},(\frg_{\C})_{-\alpha}]$ satisfying $\ll
%\alpha,H_\alpha \rr=2$; it is called the {\it coroot} associated to
%$\alpha$.  For each $\alpha$ in $R$, $i H_{\alpha}$ is in $\frt$.
%The lattice spanned by $\{H_{\alpha}, \,\alpha \in R\}$ is
%called the {\it coroot lattice} and denoted by $\check{Q}$.  We denote by
%$\frh_\R:=\sum_\alpha \R H_\alpha$, the real span of coroots.

Let $R^+$  be a choice of positive roots; denote the highest root of $R$ by $\theta$.
Let $\frh_{+} : =\{h\in \frh_\R;\; \alpha(h)\geq 0\;\text{for all}\;\alpha \in R^+\}$ be the positive chamber
(closed) in $\frh_\R$.  Let $\frA:=\{h\in \frh_{+} ;\; \theta(h)\leq 1\}$ be the fundamental alcove.
An element of $\frA$ is said to be {\it regular} if it lies strictly inside the alcove.

%$\Phi_m$ the sequence  $\Phi$ repeated $m$ times (multiplicity $m$).

%$\{\varpi_i \}_{1\leq i\leq r}$ fundamental weights, defined as the
%basis of $\frh^*$ dual to $\{H_{\alpha_i}\}_{1\leq i\leq r}$.

%Define the weight lattice $P=\{ \lambda \in \frh^* : \lambda(H_{\alpha}) \in
%\mathbb{Z},\; \forall \alpha \in R\}$. A {\it regular} weight
%$\lambda\in P^{reg}$ is such that  $ \lambda(H_{\alpha}) \neq 0$
%for all $H_\alpha$.
%Any element $\lambda$ of $P$ defines a character $e^\lambda$ of $T$ by $e^\lambda(\exp X)=\exp \lambda(X)$....OR $\exp 2i\pi\lambda(X)$?

Let $W=W(\frg_{\C},\frh)$ be the Weyl group (identified with $N_G(T)/T$).

%\begin{remark} We use complex roots (as opposed to real roots) which are
%complexification of infinitesimal roots of the pair $(\frg,\frt)$.
%Recall that for a real root $\mu:\frt \to \R$ the corresponding
%infinitesimal root is $2\pi i \mu$.
%
%Our motivation for this choice comes from our aim to
%compare two formulae (including the constants) of the symplectic volume of the moduli space; the first derived as the
%limit of the Verlinde formula which is stated in terms of complex roots, the second as given by Witten
%\cite{wi} working with infinitesimal roots.
%\end{remark}
%

\medskip

We now give the Witten formula.

\medskip

Let ${\bf a}=\{a_1,a_2,\ldots, a_s\}$ be a set  of regular elements in $\frA \subset \frh_+$.
Let $\mathcal C_j$ be the adjoint orbit of $\exp(a_j)$; we denote the collection of orbits $\mathcal C_j$ by $\mathcal C$.

Consider the function on $\frh^*$ given by
$$N_{\bf a}(\lambda)=\prod_{j=1}^s \sum_{w\in W} \varepsilon(w) e^{\ll wa_j,\lambda\rr}.$$

Let $\Phi=\Phi(G)$ be the list of positive coroots $H_\alpha$.  Define
\[
W(\Phi(G), P,g,s)({\bf a}):=\sum_{\gamma \in P^{reg}}\frac{N_{\bf a}(2i\pi \gamma)}{\prod_{H_\alpha \in \Phi}
(2i\pi \ll H_\alpha, \gamma\rr)^{2g-2+s}}.\]

The above expression is always meaningful as a generalized function of the parameters $a_j$.  If $s=0$, this formula
has to be understood as
\[
W(\Phi(G), P,g)=\sum_{\gamma \in P^{reg}}\frac{1}{\prod_{H_\alpha \in \Phi} (2i\pi \ll H_\alpha,
\gamma\rr)^{2g-2}}\]
which is meaningful if $g\geq 2$.

Interchanging the sum and the product, $N_{\bf a}(\lambda)$ may be expressed as
%$$N_{\bf a}(\lambda)=\sum_{w_{i_1}\in W}\sum_{w_{i_2}\in W} \cdots \sum_{w_{i_s}\in W}
%\varepsilon(w_{i_1})\varepsilon(w_{i_2})\cdots \varepsilon(w_{i_s})e^{\sum_{j=1}^s \ll w_{i_j}a_j,\lambda\rr}.$$
$$N_{\bf a}(\lambda)=\sum_{(w_1, w_2, \ldots, w_s)\in W^s}
\prod_{j=1}^s \varepsilon(w_j)e^{\sum_{j=1}^s \ll w_ja_j,\lambda\rr}.$$
Hence the function $W(\Phi(G), P,g,s)({\bf a})$ can be expressed as a sum over $W^s$ with signs of Bernoulli series
$\CB(\Phi_{2g-2+s},\check Q)(\sum_j w_j a_j)$.  Here, as before, $\Phi_{2g-2+s}$ means that each coroot in $\Phi$ is taken
with multiplicity $2g-2+s$.

\medskip

As is well known, the series $W(\Phi(G), P,g,s)({\bf a})$ computes the
symplectic volume of $\mathcal M(G,g,s,\mathcal C)$ up to a scalar factor, which we will give in the next section.

\bigskip

Let us now recall the normalization of the volume as the limit of the Verlinde  formula.

\medskip

We need some more notation.

%The coweight lattice $\check{P}:=\{h \in \frh:\; \alpha(h) \in \Z\;
%\text{for}\;\alpha \in R\}$.  Under the exponential map $\exp:\frh
%\to T ;\;\;H\mapsto \exp(2i\pi H)$, $\check{P}/\check{Q}$ is
%identified with

%Let $\check{P}_{\rm lg}=\{x \in \frh: \alpha(x) \in \Z \;{\rm
%for}\;\; \alpha\;\;{\rm long}\}$.

%Dual coxeter number $\check{h}:=\rho(H_\theta)+1$, $\rho$ the half
%sum of positive roots.

Let $(\;|\;)$ denote the $G$-invariant symmetric form  on $\frg_{\C}$ normalized such that
$(H_{\theta}|H_{\theta})=2$.  We will use the same notation for the
restricted form on $\frh$, and the induced form on $\frh^*$.
We call  $(\;|\;)$  the {\it basic invariant form}.  It is positive definite on the real span
$\frh_\R$, and negative definite on $\frt$.

\medskip

Let $\check{h}:=\rho(H_\theta)+1$ be the dual Coxeter number, where $\rho$ is the half
sum of positive roots.   Let $Q_{long}\subset Q$ be the lattice spanned by long roots.
The basic invariant form identifies $\frh_\R$ and ${\frh_\R}^*$; under this isomorphism the coroot lattice $\check Q$ is identified to $Q_{long}$.
Let $q$ be the index of $Q_{long}$ in $Q$, and  let $f$ be the  index of $Q$ in $P$.
Let $Z=Z(G)$ denote the center of $G$.

For a positive integer $\ell$, define the set \[P_\ell:=\{\mu \in P
\cap \frh_+^* ; \mu(H_{\theta})\leq \ell\}.\]
An element of $P_\ell$ is said to be a weight of {\it level} $\ell$.  We denote by $P'_\ell$ the subset of
$P_\ell$ consisting of elements $\mu$ satisfying $\mu(H_\theta)<\ell$ and $\mu(H_\alpha)>0$ for any simple root $\alpha$.
By definition of $\check{h}$, there is a bijection between sets $P_\ell$ and
$P'_{\ell+\check{h}}$ via $\mu \mapsto \mu+\rho$.

%The simply connected complex Lie group with Lie algebra $\frg_\C$ is the complexification $G_\C$ of the simply
%connected compact Lie group $G$.
%Then $T_\C$ is its maximal torus having $\frh$ as its Lie algebra.

Consider the maximal torus $T$ of $G$  with Lie algebra $\frt$. If $t=\exp X\in T$, with $X\in \frt$, and $\alpha$ is a
root (which takes imaginary values on $\frt$), we denote by $e^\alpha(t)=e^{\ll\alpha,X\rr}$.
Let $\Delta(t)=\prod_{\alpha \in R}(e^{\alpha}(t)-1)$.
An element of $T$ is said to be regular if $\Delta(t)\neq 0$.
Denote by $T_\ell$ the subgroup of elements $t$ of $T$ such that
$e^{\alpha}(t)$ is $\ell+\check{h}$ root of unity for each long root
$\alpha$.  We denote the set of regular elements in $T_\ell$ by $T^{\text{reg}}_\ell$.

\medskip

We now give the Verlinde formula.

\medskip

Consider the set $\underline \lambda=\{\lambda_1,\lambda_2,\ldots, \lambda_s\}$ with $\lambda_i \in P_\ell$.
Then to this collection of weights ${\underline \lambda}$ of level $\ell$, the group $G$, and a nonnegative integer $g$,
is associated a vector space $\mathcal V(G,g,s, {\underline \lambda},\ell)$ (see \cite{tuy}), called the space of
conformal blocks, whose dimension is given by the Verlinde formula $V(G,\underline \lambda,g,\ell)$:

\[
V(G,\underline \lambda,g,\ell)=(fq)^{g-1}(\ell+\check{h})^{r(g-1)}\displaystyle \sum_{t \in
T_\ell^{\text{reg}}/W}\frac{\chi_{V({\underline \lambda})}(t)}{\Delta(t)^{g-1}}.\]

Above $r$ is the rank of $G$, $V(\underline \lambda)=V_{\lambda_1}\otimes V_{\lambda_2}\otimes \cdots \otimes V_{\lambda_s}$
where $V_{\lambda_i}$ denotes the simple $\frg$ module with highest weight $\lambda_i$, and
$\chi_{V_\lambda}$ denotes the character of $V_\lambda$.  By Weyl character formula
$\chi_{V_\lambda}=J(e^{\lambda+\rho})/J(e^\rho)$ where $J(e^\nu)=\displaystyle \sum_{w\in W} \varepsilon (w)e^{w\nu}$.

\bigskip
We remark that if $\sum_{i=1}^s \lambda_i$ is not in the root lattice, then $V(G,\underline \lambda,g,\ell)$ is zero.
\medskip

Under the isomorphism given by the basic invariant form, an element $a$ lying in $\frA \subset \frh_+$ defines an
element $\hat a$ of $\mathfrak h_+^*$.
We now consider a collection $\{a_1,a_2,\ldots, a_s\}$ of rational elements in $\frA$, that is, each $a_j$ lies in the dense subset
$\frA \cap (\check{Q}\otimes \Q)\subset \frA$.
We may choose $\ell$ large enough so that each $\lambda_j:=\ell \hat a_j$ is a weight; which then lies in $P_{\ell}$.
We furthermore choose $\ell$  so that $\sum_{j=1}^s \lambda_j$ is in the root lattice and consider the space of
conformal blocks
$\mathcal V(G,g,s, {\bf \underline \lambda},\ell)$ associated to this collection
${\underline \lambda}=\{\lambda_1,\lambda_2,\ldots, \lambda_s\}$.
We can dilate simultaneously the weights $\lambda_j$ and the level $\ell$ by a factor $k$.
Then, the function
$$k\to \dim(\mathcal V(G,g,s,[k\lambda_1,k\lambda_2,\cdots,k\lambda_s],k\ell))$$ is a quasi-polynomial in $k$ of degree
$m=\dim (G)(g-1)+s|R^+|$, the complex dimension of the moduli space $\mathcal M(G,g,s,\mathcal C)$.
The volume computes the highest term of this quasi-polynomial.
More precisely,
$$\vol(\mathcal M(G,g,s,\mathcal C))=\lim_{k\to \infty} (\ell k)^{-m} dim(\mathcal V(G,g,s, [k \ell \hat a_1,
\ldots,  k \ell \hat a_s], k\ell).$$

\begin{proposition}\label{P:factor}
Let ${\bf a}=\{a_1,\cdots, a_s\}$ be a  collection of regular rational elements in $\frA$.
Let $\vol(G,g)({\bf a})$ denote the symplectic volume of the moduli
space $\mathcal M(G,g,s,\mathcal C)$. Then,
\[\vol(G,g)({\bf a})=(fq)^{g-1}\frac{|Z|}{|W|}\epsilon_G^{p(2g-2+s)}(-1)^{(g-1)|\Phi(G)|}W(\Phi(G), P,g,s)({\bf a}),\]
where $p$ is the number of short positive roots of $G$, and
$\epsilon_G=2$ for any simple Lie group except $G_2$, it is equal to
$3$ for $G_2$.
\end{proposition}

We recall that for simply laced groups $p=0$ since all roots are considered as long.

\begin{proof}
Choose $\ell$ so that each $\lambda_j:=\ell \hat a_j$  lies in $P_{\ell}$ and $\sum_{j=1}^s \lambda_j$ is a root.
Then,
\[
\vol(G,g)({\bf a})=\lim_{k \to
\infty}\frac{1}{(k\ell)^{m}}V(G, k \underline{\lambda}, g,k\ell),\] where
$m=\text{dim}(G)(g-1)+s|R^+|$ is the dimension of the moduli
space $\mathcal M(G,g,s,\mathcal C)$ and
\[
V(G, k\underline \lambda, g,k\ell)=(fq)^{g-1}(k\ell+\check{h})^{r(g-1)}
\displaystyle \sum_{t \in T_{k\ell}^{\text{reg}}/W}\frac{\chi_{V(k\underline \lambda)}(t)}{\Delta(t)^{g-1}}.\]

An element $\mu \in P_\ell$ determines a unique regular element $h_\mu \in \frA$, the image of
$\frac{\mu+\rho}{\ell+\check{h}}$ under the identification given by the basic invariant form.
Denote the image of $h_\mu$ under the exponential map by $t_\mu \in T_\C$.
The set $\{t_\mu: \mu \in P_\ell\}$ form a set of representatives for $T_\ell^{\text{reg}}/W$.
Using also the bijection between sets $P_\ell$ and $P'_{\ell+\check{h}}$ via $\mu \mapsto \mu+\rho$,

\[\begin{array}{ll}
V(G, k\underline \lambda, g,k\ell)&=(fq)^{g-1}(k\ell+\check{h})^{r(g-1)} \displaystyle
\sum_{\mu+\rho \in P'_{k\ell+\check{h}}}\frac{\chi_{V(k\underline \lambda)}(t_\mu)}
{\Delta(t_\mu)^{g-1}}\\
&=(fq)^{g-1}(k\ell+\check{h})^{r(g-1)}|W|^{-1}\displaystyle
\sum_{\mu+\rho \in W \cdot P'_{k\ell+\check{h}}}\frac{\prod_{j=1}^{s}J(e^{k\lambda_j+\rho})(t_\mu)}
{(J(e^\rho)(t_\mu))^s \Delta(t_\mu)^{g-1}}\\
&=\frac{(fq)^{g-1}}{|W|}(k\ell+\check{h})^{r(g-1)}(-1)^{|\Delta^+|(g-1)}\displaystyle
\sum_{\mu+\rho \in W \cdot P'_{k\ell+\check{h}}}\frac{\prod_{j=1}^{s}J(e^{k\lambda_j+\rho})(t_\mu)}
{(J(e^\rho)(t_\mu))^{2g-2+s}}.
\end{array}\]

The second line above follows from the fact that both $\chi_{V(\lambda)}(t)$
and $\Delta(t)$ are $W$-invariant.  The third line follows from the second by the identity
$\Delta(t)=J(e^\rho)(t)\overline{J(e^\rho)(t)}=(-1)^{|R^+|}(J(e^\rho)(t))^2$.

\medskip

We now analyze the above formula as $k$ gets large.

\medskip

The expression
\[\prod_{j=1}^{s}J(e^{k\lambda_j+\rho})(t_\mu)=\prod_{j=1}^{s}\displaystyle
\sum_{w\in W}\varepsilon(w)e^{w(k\lambda_j+\rho)}(t_\mu)\]
\[=\prod_{j=1}^{s}\displaystyle
\sum_{w\in W}\varepsilon(w)\exp\left( 2i\pi\left (\mu+\rho|\frac{w(k \ell\hat a_j+\rho)}{k\ell+\check h}\right)\right ).\]
%since $$e^{w(\lambda_j+\rho)}(t_\mu)=e^{\mu+\rho}(t_{w(\lambda_j+\rho)-\rho})
%=\exp\left( 2i\pi\left (\mu+\rho|\frac{w(k \ell\hat a_j+\rho)}{k\ell+\check h}\right)\right ).$$

Now as $k$ gets large, the expression $\exp\left( 2i\pi\left (\mu+\rho|\frac{w(k \ell\hat a_j+\rho)}{k\ell+\check h}\right)\right )$
approaches to $\exp(2i\pi(\mu+\rho|w\hat{a}_j))$. Observe also that the set $W \cdot
P'_{k\ell+\check{h}}$ approaches $P^{\text{reg}}$.  Denote an element
$\mu+\rho$ of this limiting set by $\gamma$.  Hence, $\prod_{j=1}^{s}J(e^{k\lambda_j+\rho})(t_\mu)$ approaches
$\prod_{j=1}^{s}\displaystyle
\sum_{w\in W}\varepsilon(w)e^{\ll 2i\pi\gamma,wa_j\rr}=N_{\bf a}(2i\pi \gamma)$.

\medskip

Now we analyze the denominator of the summand, $$\displaystyle \frac{1}{J(e^\rho)(t_\mu)}
=\displaystyle \frac{1}{\prod_{\alpha>0}(e^{\alpha/2}(t_\mu)-e^{-\alpha/2}(t_\mu))}
=\displaystyle \frac{1}{\prod_{\alpha>0} 2i\sin(\pi\frac{(\alpha|\mu+\rho)}{k\ell+\check h})}.$$
This expression explodes at each central vertex and the contribution from each, as $k$ gets large, is
$\frac{(k\ell+\check h)^{|R^+|}}{\prod_{\alpha>0} 2i\pi(\alpha|\mu+\rho)}$
Also observe that, for $z \in Z(G)$ both $T_\ell^{\text{reg}}$ and $\Delta(t)$ is invariant under $t\mapsto tz$.
Moreover, since $\sum_{i=1}^s \lambda_i$ is in the root lattice by construction, $\chi_{V(k\underline \lambda)}(t)$ is
also invariant.  Therefore, we may add all these equal contributions from central vertices.
(see also Remark 5.8. \cite{tw}).
Hence, we get that the expression
\[\displaystyle
\sum_{\mu+\rho \in W \cdot P'_{k\ell+\check{h}}}\frac{\prod_{j=1}^{s}J(e^{k\lambda_j+\rho})(t_\mu)}
{(J(e^\rho)(t_\mu))^{2g-2+s}}\] approaches
\[|Z(G)|(k\ell+\check{h})^{|R^+|(2g-2+s)}
\displaystyle \sum_{\gamma \in P^{\text{reg}}}\frac{N_{{\bf a}}(2i\pi \gamma)}
{\prod_{\alpha>0}(2i\pi(\alpha|\gamma))^{2g-2+s}}
.\]

By virtue of the normalization in the basic invariant form, if
$\alpha$ is a long root we have $\ll H_{\alpha},\gamma
\rr=(\alpha|\gamma)$; otherwise $\ll H_{\alpha},\gamma
\rr=\epsilon_G(\alpha|\gamma)$, where $\epsilon_G=2$ for any simple
Lie group except $G_2$, it is equal to $3$ for $G_2$. Using also
that the dimension of $G$ is $r+2|R^+|$, and that $|R^+|=|\Phi(G)|$,
we obtain
\begin{multline*}\displaystyle \lim_{k \to
\infty}\frac{1}{(k\ell)^{m}}V(G, k\underline \lambda, g,k\ell)\\
=(fq)^{g-1}\frac{|Z(G)|}{|W|}(-1)^{|\Phi(G)|(g-1)}\epsilon_G^{p(2g-2+s)}W(\Phi(G),P,g,s)({\bf
a})\end{multline*} as claimed.
\end{proof}
\bigskip

\begin{remark}
In the case of one marking, the Verlinde formula reduces to
$$
V(G,\lambda,g,\ell)=(fq)^{g-1}(\ell+\check{h})^{r(g-1)}\displaystyle \sum_{t \in
T_\ell^{\text{reg}}}\frac{e^{\lambda}(t)}{D(t)\Delta(t)^{g-1}},$$
where $D(t)=\prod_{\alpha>0}(1-e^{-\alpha}(t))$.

Following the same line of arguments as in the proof of Proposition \ref{P:factor}
we get that for $s=1$ with $\lambda=\ell \hat a$ lying in the root lattice,

\begin{equation}\label{E:1mark}
\vol(G,g)(a)=(fq)^{g-1}|Z|\epsilon_G^{p(2g-1)}(-1)^{(g-1)|\Phi|}
\sum_{\gamma \in P^{\text{reg}}}\frac{e^{\ll a,2i\pi \gamma\rr}}
{\prod_{H_{\alpha}\in \Phi}(2i\pi\ll
H_{\alpha},\gamma\rr)^{2g-1}}.\end{equation} That is,
$$\vol(G,g)(a)=(fq)^{g-1}|Z|\epsilon_G^{p(2g-1)}(-1)^{(g-1)|\Phi|}\CB(\Phi_{2g-1},\check Q)(a).$$

\end{remark}

\medskip

Let us demonstrate this with an example.
\begin{example}\label{ex:su2onemark}
We consider the moduli space of $\SU(2)$ bundles on a Riemann surface of genus $g$
with one marking.

Let $a=tH_\alpha $ be a regular element in $\mathfrak{A}$; in other words, $0<t<1/2$.
Let $\hat{a} \in \frh_\R^*$ denote the dual of $a \in \frA$ under the isomorphism given by the basic invariant form.
The Verlinde formula for $\SU(2)$ with marking
$\lambda=\ell\hat{a}=\ell t\alpha \in P_\ell$ such that $t\ell$ is a positive
integer (that is, $\lambda$ lies in the root lattice) is
{\small \[V(\SU(2),t\ell\alpha,g,\ell)=2^{g-1}(\ell+2)^{g-1}2(-1)^{g}\sum_{n=0}^{\infty}c_{g,t}(n)(\ell+2)^{2g-1-n}
\frac{B(2g-1-n,t)}{(2g-1-n)!},\]} where $c_{g,t}(n)$ is the $n^{th}$ coefficient of the Taylor series
expansion of $e^{x(g-2t)}\left(\frac{x}{e^x-1}\right )^{2g-1}$ in $x$ around zero, and $B(p,t)$ denotes the
$p^{\text{th}}$ Bernoulli polynomial in $t$.  Clearly $c_{g,t}(0)=1$.

In the expression  for $V(\SU(2),tk\ell\alpha,g,k\ell)$ the highest term in $k$ occurs when $n=0$. Hence,
\[
\begin{array}{ll}
\vol(\SU(2),g)(a)&=\displaystyle \lim_{k\to
\infty}\frac{V(\SU(2),tk\ell\alpha,g,k\ell)}{(k\ell+2)^{3g-2}}\\
&=2^{g-1}2(-1)^{g}c_{g,t}(0)\frac{B(2g-1,t)}{(2g-1)!}\\
&=2^g(-1)^{g}\frac{B(2g-1,t)}{(2g-1)!}
\end{array}\]

We now calculate the volume using Equation (\ref{E:1mark}).

\medskip

We have $P=\Z \rho$ with $\ll \rho,H_\alpha \rr=1$.  In this case, with the notation of Proposition \ref{P:factor},
$s=1$, $p=0$, $q=1$, $f=2$ and $|Z(SU(2))|=2$, hence $2^{p(2g-1)}(fq)^{g-1}|Z(G)|=2^g$.  Using the expression
(\ref{E:1mark}) of the volume in the one marking case,
{\small \[
\begin{array}{ll}
2^g(-1)^{(g-1)|\Phi|}
\displaystyle \sum_{\gamma \in P^{\text{reg}}}\frac{e^{\ll a,2i\pi \gamma\rr}}
{\prod_{H_{\alpha}\in \Phi}(2i\pi\ll H_{\alpha},\gamma\rr)^{2g-1}}&
=2^g (-1)^{g-1}\displaystyle\sum_{n\neq 0}\frac{e^{2i\pi tn}}{(2i\pi n)^{2g-1}}\\
&=2^g(-1)^g \displaystyle \frac{B(2g-1,t)}{(2g-1)!}. \end{array}\]}
Clearly we get the same formula.
\end{example}

\subsection{Volume of the moduli space as a function of the volume of $T$ and $G$}

Let us recall the formula for the symplectic volume  of the moduli
space $\mathcal M(G,g,s,\mathcal C)$ for a set of $s$ regular
conjugacy classes $\mathcal C=(\mathcal C_j)$ in $G$ as given by
E.~Witten (\cite{wi} equation 4.1.14),

{\small \begin{equation}\label{E:wformula}
\vol(\mathcal M(G,g,s,\mathcal C))=\frac{|Z(G)|}{(2\pi)^{2m}}\vol(G)^{2g-2}\vol(G/T)^s
\displaystyle \sum_{\lambda \in \text{IrrG}}\frac{\prod_{j=1}^{s}[\chi_{V_{\lambda}}(\mathcal C_j)\sqrt{\Delta(\mathcal C_j)}]}
{\dim{V_\lambda}^{2g-2+s}}\end{equation}}
where $2m$ is the real dimension of $\mathcal M(G,g,s,\mathcal C)$,
and $\text{IrrG}$ denotes the set of irreducible
representations of $G$.
Above $\vol(G)$, $\vol(G/T)$ are Riemannian volumes of $G$ and $G/T$ which we now express following Bourbaki
(Ch. IX, pages 396-411):

Choose a $\frg$-invariant scalar product.  This determines a Lebesgue measure $\mu$ on $\frg$, via identification of
$\frg$ with $\R^n$ by an orthonormal basis.
Similarly let $\tau$ be the Lesbegue measure on $\frt$ corresponding to the restriction of the scalar product on $\frt$.
We can construct from $\mu$ and $\tau$ Haar measures $\mu_G$ and $\mu_T$
on $G$ and $T$ respectively.

Since we aim to compare the volume formula in Proposition \ref{P:factor} with that of Witten in Equation (\ref{E:wformula}),
we choose the normalized Killing form as the $\frg$-invariant scalar product in the above construction,
as this was our choice in the previous section.
Then, for this choice, with respect to $\mu_G$ and $\mu_T$ constructed as above, we get that
\[\vol(G)=(fq)^{1/2}(2\pi)^{|R^+|+r}\frac{1}{\prod_{\alpha>0}(\alpha|\rho)},\;\;\;
\vol(G/T)=\frac{(2\pi)^{|R^+|}}{\prod_{\alpha>0}{(\alpha|\rho)}}.\]

\medskip

Recall from the previous section that
$$\Delta(t)=J(e^\rho)(t)\overline{J(e^\rho)(t)}=(-1)^{|R^+|}(J(e^\rho)(t))^2,$$
hence it takes positive values on a regular element $t$.
Then, parametrizing irreducible representations of $G$ with the
cone of dominant weights $P^+$,
  for $\mathcal C_j$ the adjoint orbit of $\exp(a_j)$, we may write

$$\chi_{V_{\lambda}}(\mathcal C_j)=\frac{\displaystyle \sum_{w\in W}\varepsilon(w)e^{2i\pi\ll w(\lambda+\rho),a_j\rr}}
{i^{|R^+|}\sqrt{\Delta(\mathcal C_j)}}.$$

\medskip

%\begin{remark}
%This is a remark for ourselves.  Above, the choice of square root of $-1$ is crucial.  If i don't choose this one
%Witten's formula differs from mine by factor of $-1$.  The root that i have chosen is also the choice that is done by
%Meinrenken-Woodward in their article titled `Moduli spaces of flat connections on 2 manifolds, cobordism, and Witten volume
%formulas'.  I understand from Witten that this choice has to do with choosing an orientation on the real manifold,
%and this is the compatible one, giving positive volume.
%\end{remark}

Let $d(\gamma)=\displaystyle \prod_{\alpha>0}\frac{\ll\gamma,H_\alpha\rr}{\ll\rho,H_\alpha\rr}$;
it computes the dimension of $V_{\gamma-\rho}$.

Thus
\[\displaystyle \sum_{\lambda \in \text{IrrG}}\frac{\prod_{j=1}^{s}\chi_{V_{\lambda}}({\mathcal C_j})\sqrt{\Delta({\mathcal C_j})}}
{\dim{V_\lambda}^{2g-2+s}}=\displaystyle \sum_{\lambda \in P^+}
\frac{\prod_{j=1}^{s}N_{a_j}(2i\pi(\lambda+\rho))}
{i^{s|R^+|}d(\lambda+\rho)^{2g-2+s}}.\]

Observe that the summand above is invariant under the Weyl group (both the numerator and the denominator are anti-invariant
by factor $(\text{sign}(w))^s$ for a Weyl group element $w$).  We get,
\[\displaystyle \sum_{\lambda \in \text{IrrG}}\frac{\prod_{j=1}^{s}\chi_{V_{\lambda}}(\mathcal C_j)\sqrt{\Delta(\mathcal C_j)}}
{\dim{V_\lambda}^{2g-2+s}}=
\frac{1}{|W|}\displaystyle \sum_{\gamma \in P^{reg}}\frac{\prod_{j=1}^{s}N_{a_j}(2i\pi \gamma)}
{i^{s|R^+|d(\gamma)^{2g-2+s}}}.\]

Inserting  the explicit expressions for the volume of $G$ and $G/T$ above into Equation (\ref{E:wformula}),
all $(2\pi)$ factors cancel, and combining the terms we get
\[\begin{array}{ll}
\vol(\mathcal M(G,g,s,\mathcal C))&=
\frac{|Z|}{|W|} (fq)^{g-1}
\frac{(\prod_{\alpha>0}\ll\rho,H_\alpha\rr)^{2g-2+s}}{(\prod_{\alpha>0}(\alpha|\rho))^{2g-2+s}}(-1)^{(g-1)|R^+|}
W(\Phi(G),P,g)(a)\\
&=\frac{|Z|}{|W|}
(fq)^{g-1}\epsilon_G^{p(2g-2+s)}(-1)^{(g-1)|R^+|}W(\Phi(G),P,g)(a),
\end{array}\] which is precisely the formula that we obtained in
Proposition \ref{P:factor}.

\medskip

\section{Various examples of volume calculations}

%\begin{example}
%With the notation of example \ref{ex:su2onemark} we compute the volume of the moduli space of $\SU(2)$ bundles on a
%Riemann surface of genus $g$ with $s$ markings ${\bf a}=\{a_j=t_j H_\alpha,\;j:1\cdots s\}$, all regular in $\frA$.  Using the expression
%in Proposition \ref{P:factor},
%\[\vol(\SU(2),g)({\bf a})=2^{g-1}2(-1)^{g-1}\frac{1}{2}\displaystyle \sum_{n\neq 0}
%\frac{\prod_{j=1}^s (e^{2i\pi n t_j}-e^{-2i\pi n t_j})}{(2i\pi n)^{2g-2+s}}\]
%\[=2^{g-1}(-1)^{g-1}\displaystyle \sum_{n\neq 0}\frac{\prod_{j=1}^s \sin(\pi n 2 t_j)\cdot (2i)^s}{(2i\pi n)^{2g-2+s}}\]
%\[=\frac{1}{2^{2g-2}}\frac{1}{\pi^{2g-2+s}}\displaystyle \sum_{n\neq 0}\frac{\prod_{j=1}^s \sin(\pi n 2 t_j)}{n^{2g-2+s}}\]
%\[=2\frac{1}{2^{g-1}}\frac{1}{\pi^{2g-2+s}}\displaystyle \sum_{n=1}^\infty \frac{\prod_{j=1}^s \sin(\pi n 2 t_j)}{n^{2g-2+s}},\]
%which is equation 3.18 in \cite{wi}.
%\end{example}

\begin{example}\label{exa2}
We now compute the volume of the moduli space of $\SU(3)$ bundles on a Riemann surface of genus one using the Witten series.

\medskip

Simple roots are $\alpha_1=e_1-e_2$, $\alpha_2=e_2-e_3$, and fundamental weights are
$\varpi_1=\frac{2e_1-e_2-e_3}{3}$,
$\varpi_2=\frac{e_1+e_2-2e_3}{3}$.  The positive coroots are
\[\Phi(\SU(3))=\{H_{\alpha_1}=e^1-e^2,H_{\alpha_2}=e^2-e^3,H_{\alpha_1+\alpha_2}=e^1-e^3\}\]
and $P=\Z\varpi_1\oplus \Z \varpi_2$.  Let
$\gamma=n_1\varpi_1+n_2\varpi_2$.  Then $\gamma \in P^{\text{reg}}$
if and only if $n_1\neq 0$, $n_2\neq 0$ and $n_1+n_2 \neq 0$.

Consider
\[a=a_1H_{\alpha_1}+a_2H_{\alpha_2}=a_1(e^1-e^3)+(a_2-a_1)(e^2-e^3) \in \frh_\R.\]
Suppose that $a$ is a regular element in $\mathfrak{A}$, in other
words, $2a_1-a_2>0$ $2a_2-a_1>0$ (in particular $a_1>0$ and $a_2>0$)
and $\theta(a_1H_{\alpha_1}+a_2H_{\alpha_2})=a_1+a_2<1$.

\medskip
We compute the volume using the Formula (\ref{E:1mark}).
In this case, $s=1$, $p=0$, $q=1$, $f=3$ and $|Z(SU(3))|=3$; hence, for
$g=1$, $2^{p(2g-1)}(fq)^{g-1}|Z(G)|=3$.
\[\begin{array}{ll}
\vol(\SU(3),g=1)(a)&=3\displaystyle\sum_{\gamma \in
P_{reg}}\frac{e^{2i\pi \ll a,\gamma \rr}}{\prod_{H_\alpha \in \Phi}
2i\pi \ll H_\alpha, \gamma \rr}\\
 &=3\displaystyle\sum_{n_1\neq 0, n_2\neq 0, n_1+n_2 \neq 0}
\frac{e^{2i\pi(n_1a_1+n_2a_2)}}{(2i\pi n_1)(2i\pi
n_2)(2i\pi(n_1+n_2))}
\end{array}\]

and we obtain
{\small \[\vol(\SU(3),g=1)(a)=\left\{\begin{array}{cl}
   -1/2(1+a_1-2a_2)(a_1-1+a_2)(2a_1-a_2) , & a_1\leq a_2\\
   -1/2(a_1-2a_2)(a_1-1+a_2)(2a_1-1-a_2)   & a_1\geq a_2
       \end{array}\right.\]}
\end{example}

\begin{example}
With the notation of Example~\ref{exa2}, we make similar
computations for $\SU(3)$ when genus $g=2$.

\medskip
We compute the volume employing the Formula (\ref{E:1mark}).
In this case, $s=1$, $p=0$, $q=1$, $f=3$, $|Z(\SU(3))|=3$, hence, for
$g=2$, $2^{s(2g-1)}(fq)^{g-1}|Z(G)|(-1)^{(g-1)|\Phi|}=-3^2$.

{\small \[\begin{array}{ll}
\vol(\SU(3),g=2)(a)&=-9\displaystyle\sum_{\gamma \in
P_{reg}}\frac{e^{2i\pi \ll a,\gamma \rr}}{\prod_{H_\alpha \in \Phi}
(2i\pi \ll H_\alpha,\gamma\rr)^3}\\
&=-9\displaystyle \sum_{n_1\neq 0, n_2\neq 0, n_1+n_2 \neq 0}
\frac{e^{2i\pi(n_1a_1+n_2a_2)}}{(2i\pi n_1)^3(2i\pi
n_2)^3(2i\pi(n_1+n_2))^3}
\end{array}\]}

and we obtain
{\small $$\vol(\SU(3),g=2)(a)=\left\{\begin{array}{cl}
1/40320(a_2-2a_1)(a_2-1+a_1)(-1+2a_2-a_1)P_1 , & a_1 \leq a_2\\
1/40320(a_2+1-2a_1)(2a_2-a_1)(a_2-1+a_1)P_2 , & a_1 \geq a_2
       \end{array}\right.$$}
where the polynomials $P_1$ and $P_2$ above are too long to be included in here.

\end{example}

\begin{example}\label{exb2}
We now give an example of the volume of the moduli space of  $\text{Spin}(5)$ bundles on a Riemann surface of
genus $g=1$ with one marking.

Positive roots are
$\{\alpha_1+\alpha_2=e_1,\alpha_2=e_2,\theta=e_1+e_2,\alpha_1=e_1-e_2\}$,
with associated coroots $H_{e_1}=2e^1$, $H_{e_2}=2e^2$,
$H_{e_1-e_2}=e^1-e^2$, $H_{e_1+e_2}=e^1+e^2$.

Let $a=a_1H_{\alpha_1}+a_2H_{\alpha_2}$ be a regular element in $\frA$;
in other words, $a_1>a_2$, $2a_2>a_1$, $2a_2<1$. We can express $a$
as $a=t_1 e^1+t_2 e^2$ (with $t_1=a_1$ and $t_2=2a_2-a_1$), $t_1$
and $t_2$ satisfy $t_1>t_2,t_2>0, t_1+t_2<1$.

We calculate the volume for $B_2$ and genus $g=1$ employing the Formula (\ref{E:1mark}).
In this case, $s=1$, $p=2$, $q=2$, $f=2$, $|Z(Spin5)|=2$. Hence, for
$g=1$, $$2^{s(2g-1)}(fq)^{g-1}|Z(Spin )|(-1)^{(g-1)|\Phi|}=8.$$

 We get,
\[\begin{array}{ll}
\vol(B_2,g=1)(a)&=\displaystyle \frac{1}{2}t_2(t_1-1)(t_1-1+t_2)(t_1-t_2)\\
&=(2a_2-a_1)(a_1-1)(-1+2a_2)(a_1-a_2).
\end{array}\]
\end{example}

\begin{example}
With the notation of Example \ref{exb2}, we compute the volume of the moduli space of $\text{Spin}(5)$ bundles
on a Riemann surface of genus one and two markings.

Let  ${\bf a}=\{a_1,a_2\}$, where $a_1$ and $a_2$ are  regular elements in $\frA$.
 Write $a_1=t_1 e^1+t_2 e^2$ and $a_2=u_1 e^1+u_2 e^2$. Then  the function
$\vol(B_2,1)({\bf a})$ is a piecewise polynomial function of $t_1,t_2,u_1,u_2$.
For example, choose $v_1=\frac{1}{2}e^1+\frac{1}{5 }e^2,\  v_2=\frac{1}{7}e^1+\frac{1}{9} e^2 $  and consider   $\tau({\bf v})\subset \frA\times \frA$,
the open set determined by the condition that  $a_1+w(a_2)$ is  in the same tope as
$v_1+w(v_2)$ for each element $w$ in the Weyl group of $B_2$.
Then for ${\bf a}\in \tau({\bf v})$,  we have
{\small{ $$\vol(B_2,1)({\bf a})=4 W(\Phi(B_2), P,1,2)({\bf a})=$$}\tiny{$$-{\frac {1}{60}}\,u_{{2}}{u_{{1}}}^{5}+1/6\,t_{{2}}u_{{2}}t_{{1}}{u_{{
1}}}^{3}-1/6\,t_{{2}}{u_{{2}}}^{3}t_{{1}}u_{{1}}+1/6\,{u_{{2}}}^{3}u_{
{1}}{t_{{2}}}^{2}-1/6\,u_{{2}}{t_{{2}}}^{2}{u_{{1}}}^{3}-1/12\,{u_{{2}
}}^{3}u_{{1}}{t_{{2}}}^{3}+$$}\tiny{$$1/12\,u_{{2}}{t_{{2}}}^{3}{u_{{1}}}^{3}-{
\frac {1}{60}}\,{u_{{2}}}^{5}t_{{1}}u_{{1}}+1/12\,t_{{2}}u_{{2}}{t_{{1
}}}^{3}{u_{{1}}}^{3}-1/4\,t_{{2}}u_{{2}}{t_{{1}}}^{2}{u_{{1}}}^{3}+1/4
\,t_{{2}}{u_{{2}}}^{3}{t_{{1}}}^{2}u_{{1}}+1/6\,{t_{{2}}}^{2}u_{{2}}t_
{{1}}{u_{{1}}}^{3}-$$}$$1/6\,{t_{{2}}}^{2}{u_{{2}}}^{3}t_{{1}}u_{{1}}+{
\frac {1}{60}}\,u_{{2}}t_{{1}}{u_{{1}}}^{5}-1/12\,{t_{{2}}}^{3}u_{{2}}
t_{{1}}{u_{{1}}}^{3}+1/12\,{t_{{2}}}^{3}{u_{{2}}}^{3}t_{{1}}u_{{1}}-1/
12\,t_{{2}}{u_{{2}}}^{3}{t_{{1}}}^{3}u_{{1}}+{\frac {1}{60}}\,{u_{{2}}
}^{5}u_{{1}}.
$$}

We compute some values.
\[ \vol(B_2,1)({\bf a})=\]
\tiny{\[ \left\{ \begin{array}{l@{\quad if \quad } l@{\quad and\quad} l}
\frac{141791}{372163703625}& a_1=\frac{1}{2}e_1+\frac{1}{5 }e_2& a_2=\frac{1}{7}e_1+\frac{1}{9} e_2\\
\frac{1418037104720960397931}{3721637036250000000000000000}&a_1=(\frac{1}{2}+\frac{1}{10000})e_1+(\frac{1}{5 }+\frac{1}{100000})e_2
& a_2=\frac{1}{7}e_1+\frac{1}{9} e_2
\end{array}\right.
 \]}
\end{example}

\bigskip

\section{More examples}\label{japex}
In this section we will compute some instances of the Witten volume using the formula
{\small \[\vol(G,g)({\bf a})=W(\Phi(G), P,g,s)({\bf a})2^{p(2g-2+s)}(fq)^{g-1}|Z(G)|(-1)^{(g-1)|\Phi(G)|}|W|^{-1}\]}
as given in Proposition \ref{P:factor}. We will denote by $c_{\vol}$ the factor
$$c_{\vol}:=2^{p(2g-2+s)}(fq)^{g-1}|Z(G)|(-1)^{(g-1)|\Phi(G)|}|W|^{-1}.$$
For convenience, we list values of the parameters in $c_{\vol}$ for each type of classical Lie group in Table
\ref{cosvalue}.

\begin{table}[]
\caption{Value tables}

\begin{center}
\begin{tabular}{||c|c|c|c|c|c||}
\hline \hline\
$G$-type& $q$&$f$& $p$&$|Z(G)|$&$|W|$\\
\hline\hline
$A_r$&$1$&$r+1$&0&$r+1$&$(r+1)!$\\ \hline
$B_r$&2&2&$r$&2&$2^rr!$\\ \hline
$C_r$&$2^{r-1}$&2&$r(r-1)$&2&$2^rr!$\\  \hline
$D_r$&$1$&4&$0$&4&$2^{r-1}r!$\\ \hline
\hline \hline
\end{tabular}
\end{center}
\label{cosvalue}
%\end{turn}{30}
\end{table}

\subsection{Tables of  volumes of moduli spaces}
We simply denote by $\vol(G,g)$ the Witten volume in the case of no marking, that is, when $s=0$.
We will list some values of $\vol(G,g)$ for classical Lie groups in Tables \ref{Witten table A-D} and
\ref{Witten table B-C}.
We will also list some values of the factor $c_{\vol}$ that we will need in Section \ref{comp} to
compare our computations with other numerical results in literature.

%In the tables, the column with heading $G$ refers to the simple simply connected Lie group type: for instance
%$A_r$ means of type $A$ and rank $r$.

Computations are very quick for rank less or equal to $4$ (and relatively small genus).
Beyond rank $5$, computations cannot be made within a time limit of half-hour with our method.

\newpage

\begin{sidewaystable}
\vspace{14cm}\hspace{2cm}\scriptsize
%\begin{table}
%\begin{turn}{30}
%\caption{ Witten volumes with $s=0$ type $A$, $D$}
%\begin{center}
\vspace{1cm}\hspace{2cm}\scriptsize
%\vspace{14cm}\hspace{3.5cm}\scriptsize
\begin{tabular}{|@{\barra}|c|c|c|c||}
\hline \hline
G&g&$\vol(G,g)$&$c_{\vol}$\\
\hline\hline
&2&$\frac{1}{20160}$&-$\frac{3}{2}$\\ \cline{3-4}
&3&$\frac{19}{41513472000}$&$\frac{9}{2}$\\ \cline{3-4}
&4&$\frac{1031}{189225711747072000}$&-$\frac{27}{2}$\\  \cline{3-4}
&5&$\frac{32293}{487445433460457472000000}$&$\frac{81}{2}$\\ \cline{3-4}
$A_2$&6&$\frac{27739097}{34359439770544026968653824000000}$&-$\frac{243}{2}$\\ \cline{3-4}
%\hline
&7&$\frac{29835840687589}{3031957229004108930561012205092864000000000}$ &$\frac{729}{2}$\\ \cline{3-4}
&8&$\frac{71810708985991}{598678146610235332992855225968816553984000000000}$&-$ \frac{2187}{2}$\\ \cline{3-4}
%\hline \hline
&9&$\frac{221137132669842886663}{151246026314426013297816671756013269041872371712000000000000}$&$\frac{6561}{2}$\\ \cline{3-4}
%\hline
&10&$\frac{7252062205115875364801443}{406913205782775738093852803149608268676912309073384308736000000000000}$&-$ \frac{19683}{2}$\\ \cline{3-4}

\hline \hline
&2&$\frac{23}{653837184000}$&$\frac{2}{3}$\\ \cline{3-4}
&3&$\frac{14081}{64814699109633048576000000}$&$\frac{8}{3}$\\ \cline{3-4}
$A_3$&4&$\frac{17634884778757}{11155091763154851836756404076509396992000000000}$&$\frac{32}{3}$\\  \cline{3-4}
&5&$\frac{4257463829989959473}{367215613941141638310663592914134282922173229170688000000000000}$&$\frac{128}{3}$\\ \cline{3-4}
\hline \hline
&2&$\frac{1}{27303661403504640000}$&$\frac{5}{24}$\\ \cline{3-4}
$A_4$&3&$\frac{3998447009863}{18546080754338014660769482402427552723160268800000000}$&$\frac{25}{24}$\\ \cline{3-4}
&4&$\frac{63897294036759565910707701677}{45362078256453588666977436687320229387024422032399906211071677239139256762368000000000000}$&$\frac{125}{24}$\\  \cline{3-4}
\hline \hline
$A_5$&2&$\frac{46511}{1266317152743400113929736683520000000}$&-$\frac{1}{20}$\\ \cline{3-4}
\hline \hline
&2&$\frac{1}{36}$&4\\ \cline{3-4}
&3&$\frac{1}{32400}$&16\\ \cline{3-4}
&4&$\frac{1}{14288400}$&64\\  \cline{3-4}
&5&$\frac{1}{5715360000}$& 256\\ \cline{3-4}
%\hline \hline
$D_2$&6&$\frac{1}{2240649734400}$&1024\\ \cline{3-4}
%\hline
&7&$\frac{477481}{417483460137744000000}$& 4096\\ \cline{3-4}
&8&$\frac{1}{340802824602240000}$&16384\\ \cline{3-4}
%\hline \hline
&9&$\frac{13082689}{1737399167709235430400000000}$&65536\\ \cline{3-4}
%\hline
&10&$\frac{1924313689}{99574446563452076311839744000000}$&262144\\ \cline{3-4}
\hline \hline
&2&$\frac{23}{653837184000}$&$\frac{2}{3}$\\ \cline{3-4}
$D_3$&3&$\frac{14081}{64814699109633048576000000}$&$\frac{8}{3}$\\ \cline{3-4}
&4&$\frac{17634884778757}{11155091763154851836756404076509396992000000000}$&$\frac{32}{3}$\\ \cline{3-4}
&5&$\frac{4257463829989959473}{367215613941141638310663592914134282922173229170688000000000000}$&$\frac{128}{3}$\\ \cline{3-4}
\hline \hline
&2&$\frac{68227}{1084047447508315948449792000000}$&$\frac{1}{12}$\\ \cline{3-4}
$D_4$&3&$\frac{3727283292300079}{4163527227475565552987044342329298418787470879468908707840000000000}$&$\frac{1}{3}$\\ \cline{3-4}
&4&$\frac{1}{10^{18}}\frac{5372550944533148798111597103943896132463}{401595287497255375910389668105337327730608192879514113863554560467009086512201613320691542243934208}$&$\frac{4}{3}$\\ \cline{3-4}
\hline \hline
$D_5$&2&$\frac{24136975469}{949559334622106350631397499004290699425805762560000000000}$&$\frac{1}{120}$\\ \cline{3-4}
\hline \hline
\end{tabular}
%\end{center}
\caption{ \small{Witten volumes with $s=0$ type $A$, $D$}}
\label{Witten table A-D}
%\end{turn}{30}
\end{sidewaystable}
%\end{table}
%\end{verbatim}

\newpage

\newpage

\begin{sidewaystable}
\vspace{14cm}\hspace{2cm}\scriptsize
%\vspace{13cm}\hspace{1.5cm}\scriptsize
\vspace{1cm}\hspace{2cm}\scriptsize

\begin{tabular}{|@{\barra}||c|c|c|c||}
%\caption{ Witten volumes with $s=0$ type $B$, $C$}
%\begin{center}

%\begin{tabular}{||c|c|c|c||}
\hline \hline
G&g&$vol(G,g)$&$c_{\vol}$\\

\hline\hline
&2&$\frac{1}{604800}$&16\\ \cline{3-4}
&3&$\frac{4799}{444609285120000}$&1024\\ \cline{3-4}
&4&$\frac{830287}{1058640085457339793408000}$& 65536\\  \cline{3-4}
&5&$\frac{15426208283}{26882357138658593827179724800000000}$&4194304\\ \cline{3-4}
$B_2$&6&$\frac{116091757826572267}{276425625559496603207535068049765826560000000000}$&268435456\\ \cline{3-4}
%\hline
&7&$\frac{89010391857578264959599577}{289587313199907652270725517206183101789728520798208000000000000}$ &17179869184\\ \cline{3-4}
&8&$\frac{11895139004463387894854541173}{52877369206125057710964488249573327181156553092261268684800000000000000}$&1099511627776\\ \cline{3-4}
%\hline \hline
&9&$\frac{633808384316154355255792820511421021}{3849636942831172456365873314390564274746568066276449044888380984262656000000000000000}$&70368744177664\\ \cline{3-4}
%\hline
&10&$\frac{37272442737128198664590284425995639982427091128187}{309322223503982769661277518134242125943776840468695068262378581664562780524641356087296000000000000000000}$&4503599627370496\\ \cline{3-4}

\hline \hline
&2&$\frac{19}{4302395130249216000}$&$\frac{-32}{3}$\\ \cline{3-4}
&3&$\frac{82197390067823}{9138251172376454624670846219476498015846400000000}$&$\frac{8192}{3}$\\ \cline{3-4}
$B_3$&4&$\frac{4942776168078403026588037}{259372284831214187236355852194766861663874462695850253729180680192000000000000}$&$\frac{-2097152}{3}$\\  \cline{3-4}
&5&$\frac{354908282467357148490713188145792722330626811}{10^{17}87848969140024613480397739898042540441070542192393712407652560102158237674544102913901274865860608}$&$\frac{536870912}{3}$\\ \cline{3-4}
\hline \hline
$B_4$&2&$\frac{4999}{1099339586242016873556223477363507200000}$&$\frac{16}{3}$\\ \cline{3-4}
&3&$\frac{290525115705512122541476428433}{29110138980741515852319493982928220072033854458512007601770585360017103402900354236416000000000000000}$&$\frac{16384}{3}$\\ \cline{3-4}
\hline \hline
$B_5$&2&$\frac{673653727654631}{1145031992084337393955364755047070523026083511420277985481064448000000000000}$&$-\frac{32}{15}$\\ \cline{3-4}
\hline \hline
&2&$\frac{1}{604800}$&16\\ \cline{3-4}
&3&$\frac{479}{444609285120000}$&1024\\ \cline{3-4}
&4&$\frac{830287}{1058640085457339793408000}$&65536\\ \cline{3-4}
&5&$\frac{15426208283}{26882357138658593827179724800000000}$& 65536\\ \cline{3-4}
%\hline \hline
%$C_2$&6&$\frac{4}{8752538025}$& 4194304\\ \cline{3-4}
%\hline
$C_2$&6&$\frac{116091757826572267}{276425625559496603207535068049765826560000000000}$& 268435456\\ \cline{3-4}
&7&$\frac{89010391857578264959599577}{289587313199907652270725517206183101789728520798208000000000000}$&17179869184\\ \cline{3-4}
%\hline \hline
&8&$\frac{11895139004463387894854541173}{52877369206125057710964488249573327181156553092261268684800000000000000}$& 1099511627776\\ \cline{3-4}
%\hline
&9&$\frac{633808384316154355255792820511421021}{3849636942831172456365873314390564274746568066276449044888380984262656000000000000000}$&70368744177664\\ \cline{3-4}
\hline \hline
&10&$\frac{37272442737128198664590284425995639982427091128187}{309322223503982769661277518134242125943776840468695068262378581664562780524641356087296000000000000000000}$&4503599627370496\\ \cline{3-4}
\hline \hline

&2&$\frac{19}{33612461955072000}$&$\frac{-4096}{3}$\\ \cline{3-4}
$C_3$&3&$\frac{328842292593389}{2231018352630970367351280815301879398400000000}$&$\frac{134217728}{3}$\\ \cline{3-4}
&4&$\frac{79085160601187541918557593}{1978853491449082849398466889913687604247089101378252057870336000000000000}$&$\frac{ -4398046511104}{3}$\\ \cline{3-4}
&5&$\frac{22714138337998023353982263960846269191446635633}{2094482639790168129930442330790580283190501742181628046218217852167087499488451550338298675200000000000000000}$&$\frac{144115188075855872}{3}$\\ \cline{3-4}
%&6&$\frac{907715382373293953380022951719974939549217359883260037441247}{308448834856298809146783986837335096849018498783010053926683173002375194700600447230485753482347281297795158851243212800000000000000000000}$&$\frac{-4722366482869645213696}{3}$\\ \cline{3-4}
\hline \hline
&2&$\frac{2668859}{2236612113936395006421353103360000000}$&$\frac{4194304}{3}$\\ \cline{3-4}
$C_4$&3&$\frac{297521088216082657235423644965697}{433774873327337441627971738322499693513421035684824103143372913599269142790144000000000000000}$&$\frac{1125899906842624}{3}$\\ \cline{3-4}
%&4&$\frac{1375385291414363300098964738315680236692382190415485407848427673}{3425457990060264160597286183838711204634391772033830640239531020946782259530421474671832406384624545144768289039716273811767898931200000000000000000000000}{6274926367144615248599838564145895745790753013742408029118040007297016976753150208135805347561472}$&$\frac{302231454903657293676544}{3}$\\ \cline{3-4}
\hline \hline
$C_5$&2&$\frac{5395475362754527}{1066394142885075317654167073822645957605991057511677952000000000000}$&$\frac{-274877906944/15}{15}$\\ \cline{3-4}
\hline \hline
\end{tabular}
%\end{center}
\caption{\small{Witten volumes with $s=0$ type $B$, $C$}}
\label{Witten table B-C}
%\end{turn}{30}
%\end{table}
\end{sidewaystable}

%We can also compute the value for $C_3$ and $g=6$, $C_4$ and $ g=4$, but the values is too long to be put in the table.
%$C_3, g=6:=\frac{907715382373293953380022951719974939549217359883260037441247}{308448834856298809146783986837335096849018498783010053926683173002375194700600447230485753482347281297795158851243212800000000000000000000}$
%$C_4,g=4:\frac{1375385291414363300098964738315680236692382190415485407848427673}{3425457990060264160597286183838711204634391772033830640239531020946782259530421474671832406384624545144768289039716273811767898931200000000000000000000000}{6274926367144615248599838564145895745790753013742408029118040007297016976753150208135805347561472}$

\subsection{Comparison results}\label{comp}
In this section we compare some of our computations of $\vol(G,g)$ with that of Komori-Matsumoto-Tsumura
(\cite{kmt0},\cite{kmt1},\cite{kmt2}). The setting is as follows.

As before,  $G$ is a simple, compact Lie group of rank $r$.
We do not assume that $G$ is simply-connected. Let $L$ be the weight lattice of $G$.
Let $P$ be the weight lattice of the simply connected group covering $G$ and let
$Q$ be its root lattice.  Then, $Q\subset L\subset P$. Let $P^+$ be the `cone' of dominant weights, and let
$L^+=L\cap P^+$.

Let ${\bf s}=[s_\alpha]$ be  a sequence of real variables indexed by  the positive roots $R^+$.
For  $v\in \frh_\R$,  Komori-Matsumoto-Tsumura introduced
$$\zeta({\bf s},v,G)=\sum_{\gamma\in\rho+L^+} e^{2i\pi\ll v,\gamma\rr}\prod_{\alpha\in R^+}\frac{1}{\ll  \gamma,H_\alpha\rr^{s_\alpha}}.$$
If $G$ is simply connected, then $L=P$, and we may denote $\zeta({\bf s},v,G)$
by $\zeta({\bf s},v,\frg) $, or for the Lie algebra $\frg$ of type $X_r$ by $\zeta({\bf s},v,X_r) $ as in \cite{kmt2}.

\begin{example}
Consider the simply connected group $G=\SU(4)$; its positive roots are
$$[e_1-e_2, e_2-e_3, e_3-e_4, e_1-e_3, e_2-e_4, e_1-e_4].$$

The cone of dominant weights is the simplicial cone  generated by fundamental weights $\omega_1,\omega_2$ and
$\omega_3$ that are dual to simple coroots $e^1-e^2$, $e^2-e^3$ and $e^3-e^4$ respectively.
Then, if we order the exponents ${\bf s}=[s_i]$ with respect to the order of the roots as given above,

$$\zeta({\bf s},v,\SU(4))=\zeta({\bf s},v,A_3)$$
$$=\sum_{m_1=1}^{\infty} \sum_{m_2=1}^{\infty} \sum_{m_3=1}^{\infty}
\frac{e^{2i\pi \ll v,m_1\omega_1+m_2\omega_2+m_3\omega_3\rr}}
{ m_1^{s_1}
 m_2^{s_2} m_3^{s_3} (m_1+m_2)^{s_4} (m_2+m_3)^{s_5} (m_1+m_2+m_3)^{s_6}}.$$
\end{example}

\bigskip
\medskip

The series $\zeta({\bf s},v,G)$ converges when the exponents $s_{\alpha}$ are sufficiently large.
It can be shown that $\zeta({\bf s},v,G)$ can be continued as a meromorphic function of ${\bf s}$.
Let $S=\sum s_\alpha$.
Suppose $s_\alpha$ are the same for all short roots, respectively for all long roots, and
both are equal to positive even integers (that are not necessarily the same positive even integers).
Then $(2\pi)^{-S}\zeta({\bf s},0,G)$ is rational.
Indeed, using the invariance of the sum under the Weyl group $W$,
$(2\pi)^{-S}\zeta({\bf s},0,G)$ is proportional to a Bernoulli series (with repetition of coroots in $\Phi$
matching the exponent data) which is obtained by summing over
all the regular elements of the full lattice $L$.  More precisely,
\begin{equation}
\displaystyle \frac{\zeta({\bf s},0,G)}{(2\pi)^S}=|W|^{-1}{i^S}\sum_{\gamma\in L_{reg}} \frac{1}
{\prod_{\alpha\in R^+}  (2i\pi \ll \gamma,H_\alpha\rr)^{s_\alpha}},
\end{equation}
where the series on the right hand side is a multiple Bernoulli series which has (in the case that it converges absolutely)
rational value.

\medskip

If all $s_\alpha$ are equal to an even integer $2k$, we denote the sequence ${\bf s}=[s_\alpha]$ by ${\bf s}_{2k}$.
Then, for exponents ${\bf s}_{2k}$, and $G$  simply connected,
we may compute $\zeta({\bf s}_{2k},0,G)$ using the Witten volume formula for $g=k+1$,
\begin{equation}\label{zeta}
\begin{split}
\zeta({\bf s}_{2k},0,G)=|W|^{-1}(2\pi)^{2k |R^+|}(-1)^{k|R^+|} W(\Phi(G), L ,k+1)(0)
\\
=|W|^{-1}(2\pi)^{2k |R^+|} (-1)^{k |R^+|}\frac{1}{c_{\vol}} \vol(G,k+1)({0}).\\
\end{split}
\end{equation}
Above $\Phi(G)$ denotes the set of positive coroots as before.
Thus we can  use the values of the volume listed in the tables of the previous section  to compute some instances
of the series $\zeta({\bf s}_{2k},0,G)$.

\medskip
We now demonstrate some computations of $\zeta({\bf s}_{2k},0,G)$.

\subsubsection{Examples of type $A_r$}
Let $n=r+1$.  We consider the simply connected group $G=\SU(n)$.
If we write $N=|R^+|=\frac{n(n-1)}{2}$, then Equation (\ref{zeta})  is
$$\zeta({\bf s}_{2k},0,A_r)=(-1)^{kN}(2\pi)^{2kN}\frac{1}{n!}\frac{1}{c_{\vol}}\vol(\SU(n),k+1)(0),$$
where
$c_{\vol}=n^{k+1}(-1)^{k\frac{n(n-1)}{2}}\frac{1}{n!}$.

Thus we can recover, the values of $\zeta({\bf s}_{2k},0,A_r)$ for $n=3,4,5,6$ using Table \ref{Witten table A-D}.
%multiplying
%$(-1)^{kN}(2\pi)^{2kN}\frac{1}{n!}$ by the values of the volume listed in the third column and the value of
%the inverse of $C_{\vol}$ listed  in the fourth column of  the Table \ref{Witten table A-D}.

For instance, if $n=3$ (that is $r=2$), and $k=1$, then  we have $N=3$, $\vol(\SU(3),2)(0)=\frac{1}{20160}$ and
$c_{\vol}=-3/2$,
%first line of Table \ref{Witten table A-D},
and we obtain
$$\zeta({\bf s}_{2},0,A_2)=(2\pi)^{6}(-1)^3\frac{1}{3!}\frac{1}{9(-1)^3\frac{1}{3!}}\frac{1}{20160}
=\pi^{6}\frac{1}{2835}$$ as in \cite{kmt3} equation 7.11.

\medskip

We give one other example whose parameters are not contained in the tables.
Consider $n=4$, $k=5$. Then, $N=6$ and
{\tiny{$$\zeta({\bf s}_{2k},0,A_3)=(2\pi)^{60}\frac{1393614066290742513412310095846}{58203152419058513584890890509712229288124323632762771449711578369140625}$$}}

\subsubsection{Examples of type  $B_r$, $C_r$ and $D_r$}
%We obtain Prop.~2.1 in \cite{kmt2}.

For root systems of type $B_r$ and $C_r$,  the number of positive
roots is $N=r^2$ and the order of the Weyl group is $|W|=r!2^r$. For
example, for $B_r$ when all exponents $s_\alpha=2k$,
$$\zeta({\bf s}_{2k},0,B_r)= \frac{1}{r!2^r}(2\pi)^{2kN}(-1)^{kN}\mathcal B(\CH_r^{BC},\check{Q}_B,g^B_{{\bf s}_{2k}})(0)$$
%$$\zeta({\bf s}_{2k},0,C_r))= \frac{1}{r!2^r}(2\pi)^{2kN}(-1)^{kN} \mathcal B(\CH_r^B,\check{P_C},g^C_{{\bf s}_{2k}})(0)$$

\medskip

Explicitly for $C_2$, fundamental weights are $e_1$ and $e_1+e_2$,
and positive roots are $[e_1-e_2, 2e_2, 2 e_1, e_1+e_2]$.  We
consider the multiple zeta series
$$\zeta([s_1,s_2,s_3,s_4],0,C_2)=\sum_{m=1}^\infty\sum_{n=1}^\infty \frac{1} {m^{s_1}n^{s_2}(m+n)^{s_3}(m+2n)^{s_4}},$$
where we order the exponents with respect to the order given in the list of roots above.
In the particular case that all $s_i=2$,  using
$|W|=8$ and values in Table \ref{Witten table B-C}, we find that for ${\bf {s_2}}=[2,2,2,2]$
$$\zeta({\bf {s_2}},0,C_2)=\frac{1}{8}(-1)^4(2\pi)^{8}\frac{1}{16}\frac{1}{604800}=\frac {1}{302400}\,{\pi }^{8},$$
which is the equation (7.23) of \cite{kmt3}.

%For $B_2$ we compute similarly
%$\zeta([s_1,s_2,s_3,s_4],0,B_2)=$\zeta([s_1,s_2,s_3,s_4],0,C_2)=\sum_{m=1}^\infty\sum_{n=1}^\infty \frac{1} {m^{s_1}n^{s_2}(m+n)^{s_3}(m+2n)^{s_4}}$$

\medskip

We also give an example of $D_4$ with all exponents equal to $6$ (that is
$k=3$ and ${\bf {s_6}}=[6,6,6,6,6,6,6,6,6,6,6,6]$ ).
$$\zeta({\bf s_6},0,D_4)=$${\tiny{$$\frac{5372550944533148798111597103943896132463}{21770524158223250767856810653451043131130341521323218291199402843808716814637088000000000000000000}\pi^{72}$$}}

\bigskip
It is also possible to compute $\zeta({\bf s},0,G)$ when the exponents in the list ${\bf s}=[s_1,s_2,s_3,s_4]$ are
different positive even integers for short and long roots.  We conclude with one example of this kind.

\medskip

Consider the list of exponents $[2,4,4,2]$ corresponding to the list of positive roots $[e_1-e_2, 2 e_2, 2 e_1, e_1+e_2]$
of $C_2$.
Then,
$$\zeta([2,4,4,2],0, C_2)=\sum_{m=1}^\infty\sum_{n=1}^\infty \frac{1} {m^2n^4(m+n)^4(m+2n)^2}=\pi^{12}\frac{53}{6810804000},$$
which coincides with equation (4.30) of \cite{kmt3}.

%WHAT IS THE NOTATION ?? FOR $s$. NOT THE SAME THAT WHAT YOU USED BEFORE.
%THEN EXPLAIN THE NOTATION.
%\marginpar{II don't understand :in the above writing  m1 correspond to e1-e2,
%m2 to e2, m1+m2 to e1 and m1+2m2 to e1+e2, so I compute with exponent 4 for e1 e2 and 2 for the other. IS NOT THIS?}

%Here is the explanation. I use the exponents as the jap. For us the order is given as [e1-e2,e1+e2,e1,e2 and so in
%computing with maple the exponents are $[2,2,4,4]$
%You can look the example computed in the file examples I enclose
%\marginpar{We can change if you want}

\medskip

\subsection{Some multiple zeta values}
Let $k$ be a positive integer. Consider the multiple zeta series
{\small $$\zeta_r(2k,2k,\ldots,2k):=
\sum_{m_1=1}^{\infty} \sum_{m_2=1}^{\infty}\cdots \sum_{m_r=1}^{\infty} \frac{1}{m_1^{2k}}
\frac{1}{(m_1+m_2)^{2k}}\cdots
\frac{1}{(m_1+m_2+\cdots+m_r)^{2k}}.$$}

Following \cite{kmt0}, we want to demonstrate how the above series can be computed
using the Bernoulli series ${\mathcal B}(\mathcal H_r^{BC},\check{Q}_C,g_{\bf s}^C)(0)$ for the root system of type
$C_r$, where the exponents
${\bf s}=[s_\alpha]$ are taken to be $0$ for long positive roots, and $2k$ for short
positive roots.
%This series of exponents is invariant by the action of the Weyl group.
Using the invariance of the sum under the Weyl group, which is of order $2^r r!$ for $C_r$, we may write
$$\mathcal B(\CH_r^{BC},\check{Q}_C, g_{\bf s}^C)(0)=2^r r! \sum_{\gamma\in (P_C^+)_{reg}} \frac{1}{\prod_{\alpha>0} (2i\pi \ll H_\alpha,\gamma \rr)^{s_\alpha}}.$$

A dominant integral regular weight $\gamma \in {(P_C^+)}_{reg}$ is of the form $\gamma=\sum_{i=1}^r m_i \omega_i$
with $m_i\geq 1$ (as before $\omega_i$ denotes the fundamental weights).  Also recall that the root system of type $C_r$
admits $r$ long roots $\{{2e_i}\}_{1\leq i\leq r}$, with corresponding (short) coroots  $\{H_{2e_i}=e^i\}_{1\leq i\leq r}$.
If we express $H_{2e_i}=e^i=(e^i-e^{i+1})+ (e^{i+1}-e^{i+2})+\cdots+ e^r$, then
$\langle H_{2e_i},\gamma\rangle=m_i+m_{i+1}+\cdots +m_r.$
Thus,
$$\zeta_r(2k,2k,\ldots,2k)=(-1)^{kr}(2\pi)^{2k r}\frac{1}{2^r r!} \mathcal B(\CH_r^{BC},\check{Q}_C, g_{\bf s}^C)(0).$$

For example, $\zeta_2(4,4)=\frac{\pi^8}{113400}$, $\zeta_5(4,4,4,4,4)=\frac{\pi^{20}}{548828480360160000}$,
$\zeta_5(6,6,6,6,6)=\frac{\pi^{30}}{1347828286825972065254765625}$.

\bigskip
\bigskip

\section{Appendix: Szenes formula}\label{szf}

Let $\CH$ be an arrangement of hyperplanes compatible with a lattice $\Lambda$.
Let $g\in \CR_{\CH}$.  Consider
$$\CB(\CH,\Lambda,g)(v)=\sum_{\gamma\in \Gamma_{reg}(\CH)}g(2i\pi\gamma) e^{2i\pi \la v,\gamma\ra}.$$

This function (a generalized function on $V$)  coincide with a polynomial function
$\CB(\CH,\Lambda,g,\tau)$ on a tope $\tau$ (see Proposition \ref{prop:polyn}).
The piecewise polynomial function $P(\CH,\Lambda,g)$ has been defined in  Definition \ref{def:sz}.
Following Szenes \cite{sze1}, we prove the following formula.

\begin{theorem}(Szenes)\label{theo:mainformula}
Let $g\in \CR_{\CH}$.
On $V_{reg}(\CH,\Lambda)$ we have the equality
$$\CB(\CH,\Lambda,g)= P(\CH,\Lambda,g).$$
\end{theorem}

We recall that, for $f\in S_{\mathcal H}$,
$$\SZ^{\Lambda}(v)(f)(z)=\sum_{\gamma\in \Gamma}  f(2i\pi\gamma-z) e^{\langle v,2i\pi\gamma-z\rangle },$$
and $P(\CH,\Lambda,g)(v)$ is the trace on $S_{\CH}$ of the operator  $A(v,g):\CS_\CH\to \CS_\CH$ defined by
\begin{equation}\label{eq:AVg}
f(z)\mapsto {\bf R}(e^{\ll z,v\rr}g(z) (\SZ^{\Lambda}(v)f)(z)).
\end{equation}
Here ${\bf R}: R_{\mathcal H}\to S_{\mathcal H}$ is the total residue.

We first consider the one dimensional case where $V=\R$, and $\Lambda=\Z$.
Here $\CH=\{0\}$, with equation $z=0$.
The topes are the intervals $]-n,n+1[$, and the space $\CS_\CH$ is one dimensional with basis $f_\sigma=\frac{1}{z}$.

Let $\tau=]0,1[$.
Assume $v\in \tau$ so that $[v]=0$.  If we consider $g(z)=\frac{1}{ z^k}$, the formula to be proven is
\begin{equation}\label{okone}
\sum_{n\neq 0}\frac{e^{2i\pi n v}}{(2i\pi n)^k}=\Res_{z=0} (\frac{1}{z^k} e^{zv}(\SZ^{\Lambda}(\tau)f_\sigma)(z)).
\end{equation}
As $\SZ^{\Lambda}(\tau)(f_\sigma)(z)=\frac{1}{1-e^z}$  (see Example \ref{ex:dim1Eis}),
we have  thus to verify that
$$\sum_{n\neq 0}\frac{e^{2i\pi n v}}{(2i\pi n)^k}=\Res_{z=0}( \frac{1}{z^k} e^{zv}\frac{1}{1-e^z}).$$

The poles of the function $\frac{1}{1-e^z}$ consist of the elements $2i\pi n$, with $n\in \Z$, and , when $k\geq 0$,
the equality above follows from the residue theorem in one variable. If $k<0$, both sides vanish
(the left hand side gives a generalized function  supported on $\Z$, the right hand side has no poles).

Szenes formula generalizes this result in higher dimensions, which we aim to demonstrate below.

\medskip

\begin{proof}

We first remark that using both the comparison formulae (\ref{lem compareP }) and (\ref{lem:secondcompare}) over
commensurable lattices, it suffices to prove the equality for any lattice $\Lambda$ (compatible with $\CH$) of our choice.

We will prove Theorem  \ref{theo:mainformula} by the standard `deletion-contraction' argument on arrangement of
hyperplanes.

Choose a set $\Phi^{eq}$ of equations for $\CH$. For $\phi\in \Phi^{eq}$,  we consider the following two arrangements:

$\bullet$  $\mathcal H'=\mathcal H\setminus H_\phi$.

$\bullet$ $\mathcal H_0=\{H\cap H_\phi, H\in \mathcal H'\}$, the trace of the arrangement $\mathcal H'$ on $H_\phi$.

Consider the vector space $V_0:=V/\R \phi$, let $p:V\to V_0$ be the projection.
The dual space $U_0$ of the vector space  $V_0$ is the hyperplane $H_\phi$.

We now compare the spaces $\CS_{\CH}$,
$\CS_{\CH_0}$ and $\CS_{\CH'}$.
\begin{definition}
We say that a function $f\in \CM_{\mathcal H}$ has at most  a simple pole along the hyperplane $\phi=0$ if
$\phi f \in \CM_{\mathcal H'}$.
In this case,  we define $res_\phi f\in \CM_{\mathcal H_0}$ by
$res_\phi f= (\phi f)|_{H_\phi}$
\end{definition}

In other words, the meromorphic function $f$ has at most a simple pole on $H_\phi$ if the denominator of $f$ contains
the factor $\phi$ at most once. Then we multiply $f$ by $\phi$, eliminating $\phi$ from the denominator of $f$, and we
can restrict $\phi f$ to $\phi=0$. This operation kills the functions $f$ having no poles of $\phi=0$.

If $f=\frac{1}{\phi}f'$ with $f'\in \CM_{\CH'}$, then
\begin{equation}\label{eq:resphi}
res_\phi {\rm {\bf R}} f={\rm {\bf R}} res_\phi f.
\end{equation}

This is easy to verify using for example a decomposition of $f'$ with denominator on a set of independent hyperplanes
(see Lemma \ref{ind}).

The map $res_\phi$ is well defined on $\CS_{\mathcal H}$, as elements in $\CS_{\mathcal H}$ have at most a simple pole on $\phi=0$.
It is easy  to prove that
we have the exact sequence
\begin{equation}\label{eq:exactseq}
\begin{CD} 0 @> >>\CS_{\mathcal H'}@>i>> \CS_{\mathcal H} @>res_\phi>> \CS_{\mathcal H_0}@>  >> 0 .\end{CD}
\end{equation}

\medskip

Let $v\in V_{reg}(\Lambda,\CH)$. Its projection $v_0=p(v)$  belongs to $V_{reg}(\Lambda_0,\CH_0)$.

\begin{lemma}\label{lem:resSZ}
Let $v\in V_{reg}(\Lambda,\CH)$ and  $f\in \CS_{\mathcal H}$. Then
$$res_\phi \SZ^{\Lambda}(v)(f)= -\SZ^{\Lambda_0}(v_0)(res_\phi f),$$
with $v_0=p(v)$.
\end{lemma}
\begin{proof}
We have
$$\SZ^{\Lambda}(v)(f)(z)=\sum_{\gamma\in \Gamma}  f(2i\pi\gamma-z) e^{\langle v,2i\pi\gamma-z\rangle }.$$
If $\gamma$ is  such that $\ll \phi,\gamma\rr\neq 0$, then the term $f(2i\pi\gamma-z)$ has no pole on $\phi=0$. Thus we obtain, for $z\in H_\phi$,
\begin{eqnarray*}
res_\phi \SZ^{\Lambda}(v)(f)(z)&=&\sum_{\gamma\in \Gamma, \ll \gamma,\phi\rr =0} (\phi(z) f(2i\pi\gamma-z))|_{H_\phi} e^{ \langle v_0,2i\pi\gamma-z\rangle }\\
&=&-\sum_{\gamma\in \Gamma, \ll \gamma,\phi\rr=0}  \phi(2\pi \gamma-z) f(2i\pi\gamma-z)|_{H_\phi} e^{\langle v_0,2i\pi \gamma-z\rangle }.
\end{eqnarray*}
\end{proof}

Let $g\in \CR_{\CH'}$, and let $g_0$ be its restriction to $H_\phi$.
Then the  operator $A(v,g)$ leaves $\CS_{\CH'}$ stable.

If $F$ has at most a simple pole on $\phi=0$, then $gF$ also has at most a simple pole
on $\phi=0$, as $g$ has no pole on $\phi=0$. Thus the maps in the diagram below are well defined. Its commutativity follows from
Lemma \ref{lem:resSZ}.

\begin{lemma}
Let $g\in \CR_{\CH'}$.
Then the following diagram is commutative.
\begin{equation}\label{eq:cd}\begin{CD}
      0 @>>>  \CS_{\CH'} @>>> \CS_{\CH}  @>>> \CS_{\CH_0} @>>>  0 \\
      @. @VVA(v,g)V  @VVA(v,g)V  @VV-A(v_0,g_0)V @. \\
      0 @>>>  \CS_{\CH'} @>>>  \CS_{\CH} @>>> \CS_{\CH_0} @>>>  0
\end{CD}\end{equation}
\end{lemma}

We are now ready to prove Theorem \ref{theo:mainformula} by induction on the number of hyperplanes in $\CH$.
If there are less than $r$ hyperplanes, then $\CS_\CH=\{0\}$, the generalized function $\CB(\CH,\Lambda,g)$ is
supported on affine walls, so both sides of the equation of Theorem \ref{theo:mainformula} vanish.

Assume that $\CH$ consists of $r$ independent hyperplanes intersecting on $\{0\}$.
Changing the lattice $\Lambda$, we can eventually assume that $\Lambda$ is the lattice generated by the equations
$\phi_k$ of the hyperplanes.  Then, the theorem follows from Formula (\ref{okone}) in the
one dimensional case.

Assume that $\CH$ have more than  $r$ hyperplanes.
Then by the Lemma \ref{ind}, we can write a function in $\CR_{\CH}$  as a sum of functions $g$ whose poles lie on an
independent subset of hyperplanes of $\CH$, thus in number less or equal to $r$.
Thus $\CR_\CH$ is linearly generated by functions $g$ such that  some equation $\phi\in \Phi^{eq}$  is not  a pole of $g$.
We consider such a couple $(g,\phi)$ and the arrangements $\CH'$ and $\CH_0$ associated to $\phi$ by deletion and
contraction.  The function $g$ is in $\CR_{\CH'}$.

Let $g_0 \in \CR_{\mathcal H_0}$ be the restriction of $g$ to $H_\phi$.  Thus $\mathcal{B}(\mathcal H_0,\Lambda_0,g_0)$ is
a generalized function on $H_\phi^*=V/\R \phi$ and $p^*\mathcal{B}(\mathcal H_0,\Lambda_0,g_0)$ is a function on $V$
(constant in the direction $\phi$).

We have the following recurrence relation for the function (eventually generalized)
$\mathcal{B}(\mathcal H,\Lambda,g)$ associated to an element $g\in
R_{\mathcal H'}$.

\begin{proposition} \label{prop:recug}
If $g\in \CR_{\mathcal H'}$, then
$$\mathcal{B}(\mathcal H,\Lambda,g)=\mathcal{B}(\mathcal H',\Lambda,g)-p^*
\mathcal{B}(\mathcal H_0,\Lambda_0,g_0).$$
\end{proposition}

This is  clear. Indeed the set $\Gamma_{reg}(\CH')$ is larger than $\Gamma_{reg}(\CH)$ as it may contain also elements
$\gamma$ with $\ll\gamma,\phi\rr=0$.
This additional summation gives rise to the term
$\mathcal{B}(\mathcal H_0,\Lambda_0,g_0).$

\bigskip

Let  $v\in V_{reg}(\CH,\Lambda)$.
As $P(\CH,\Lambda,g)(v)$ is the trace of the operator $A(v,g)$ defined in (\ref{eq:AVg}), the commutativity of the
diagram (\ref{eq:cd}) above implies that
$$P(\mathcal H,\Lambda,g)(v)=P(\mathcal H',\Lambda,g)(v)-
P(\mathcal H_0,\Lambda_0,g_0)(v_0).$$
Comparing with Proposition \ref{prop:recug}, we see by induction that Szenes formula holds.
\end{proof}

\end{document}